\documentclass[reqno,11pt]{amsart}

\usepackage{latexsym}
\usepackage{float}
\usepackage{amssymb}
\usepackage{amsmath}
\usepackage{amsthm}
\usepackage{mathrsfs}
\usepackage{euscript}
\usepackage[cmtip,all]{xy}

\usepackage{bbold}
\usepackage{tabu}
\usepackage{enumerate}

\usepackage{framed}

\usepackage{graphicx}

\usepackage{stackrel}
\usepackage[new]{old-arrows}

\usepackage{tikz}
\usetikzlibrary{intersections}
\usetikzlibrary{backgrounds}

\usepackage{tkz-euclide}

\makeatletter
\g@addto@macro \normalsize {%
 \setlength\abovedisplayskip{7pt}%
 \setlength\belowdisplayskip{6pt}%
}
\makeatother

\newtheorem{thm}[equation]{Theorem}
\newtheorem{lem}[equation]{Lemma}
\newtheorem{cor}[equation]{Corollary}
\newtheorem{prop}[equation]{Proposition}

\newtheorem{prob}[equation]{Problem}

\theoremstyle{remark}
\newtheorem{rem}[equation]{Remark}

\newtheorem{obs}[equation]{Observation}

\newtheorem{defn}[equation]{Definition}

\numberwithin{equation}{section}

\newcommand{\gb}{\beta}
\newcommand{\ga}{\alpha}

\newcommand{\gL}{\Lambda}

\newcommand{\gD}{\Delta}

\newcommand{\eps}{\varepsilon}

\newcommand{\fa}{{\mathfrak a}}             
\newcommand{\fb}{{\mathfrak b}}
\newcommand{\fg}{{\mathfrak g}}
\newcommand{\fh}{{\mathfrak h}}
\newcommand{\fk}{{\mathfrak k}}
\newcommand{\fl}{{\mathfrak l}}
\newcommand{\fm}{{\mathfrak m}}
\newcommand{\fn}{{\mathfrak n}}

\newcommand{\fp}{{\mathfrak p}}

\newcommand{\fs}{{\mathfrak s}}
\newcommand{\fy}{{\mathfrak y}}

\newcommand{\f}{\mathfrak}

\newcommand{\R}{\mathbb{R}}          
\newcommand{\C}{\mathbb{C}}          
           
\newcommand{\Z}{\mathbb{Z}}

\newcommand{\ad}{\mathrm{ad}}
\newcommand{\Ad}{\mathrm{Ad}}
\newcommand{\Cal}{\mathcal}

\newcommand{\Hom}{\operatorname{Hom}}

\renewcommand{\Im}{\mathrm{Im}}
\newcommand{\Ind}{\mathrm{Ind}}
\newcommand{\IP}[2]{\langle#1 , #2\rangle}     %inner product.

\newcommand{\spn}{\text{span}}

\newcommand{\Tr}{\text{Tr}}

\newcommand{\Pol}{\mathrm{Pol}}
\newcommand{\poly}{\mathrm{poly}}

\newcommand{\To}{\longrightarrow}

\newcommand{\Diff}{\mathrm{Diff}}
\newcommand{\Irr}{\mathrm{Irr}}

\newcommand{\triv}{\mathrm{triv}}
\newcommand{\fin}{\mathrm{fin}}
\newcommand{\sym}{\mathrm{sym}}

\newcommand{\D}{\Cal{D}}

\newcommand{\wH}{\widetilde{H}}

\newcommand{\wT}{\widetilde{\Cal{T}}}

\newcommand{\wlambda}{\widetilde{\lambda}}

\newcommand{\wk}{\widetilde{k}}
\newcommand{\wm}{\widetilde{m}}

\newcommand{\diag}{\mathrm{diag}}

\newcommand{\id}{\mathrm{id}}
\newcommand{\pr}{\mathrm{pr}}

\newcommand{\proj}{\mathrm{proj}}
\newcommand{\Proj}{\mathrm{Proj}}

\newcommand{\sgn}{\mathrm{sgn}}

\newcommand{\std}{\mathrm{std}}

\newcommand{\RP}{\mathbb{R}\mathbb{P}}

\newcommand{\Ker}{\mathrm{Ker}}

\newcommand{\Cartan}{\mathrm{Cartan}}
\newcommand{\PRV}{\mathrm{PRV}}

\newcommand{\Exterior}{\mathchoice{{\textstyle\bigwedge}}%
    {{\bigwedge}}%
    {{\textstyle\wedge}}%
    {{\scriptstyle\wedge}}}    

\providecommand*{\donothing}[1]{}

\makeatletter
\@namedef{subjclassname@2020}{%
  \textup{2020} Mathematics Subject Classification}
\makeatother

\begin{document}

\baselineskip=16pt
\tabulinesep=1.2mm

%%%%%%%%%%%%%%%%%%%%%%%%%%%%%%%%%%%%%%%%%%%%%%%

\title[]{On the intertwining differential operators between vector bundles 
over the real projective space of dimension two}

\dedicatory{Dedicated to the memory of Joseph A.\ Wolf}

\author{Toshihisa Kubo}
\author{Bent {\O}rsted}

\address{Faculty of Economics, 
Ryukoku University,
67 Tsukamoto-cho, Fukakusa, Fushimi-ku, Kyoto 612-8577, Japan}
\email{toskubo@econ.ryukoku.ac.jp}

\address{Department of Mathematics, 
Aarhus University,
Ny Munkegade 118 DK-8000 Aarhus C Denmark}
\email{orsted@imf.au.dk}
 %Representations of Lie algebras and Lie superalgebras, algebraic theory (weights)
%05B20, %Combinatorial aspects of matrices (incidence, Hadamard, etc.)
%33C05, %Classical hypergeometric functions, 2F1
%33C45, %Orthogonal polynomials and functions of hypergeometric type 
%33E30, %Other functions coming from differential, difference and integral equations
%}

%differential symmetry breaking operator,
%$K$-type formula.
%Peter--Weyl theorem for solution spaces,
%small representations,
%$K$-type solution,
%polynomial solution,
%hypergeometric differential equation,
%Heun's differential equation,

\date{\today}

\maketitle

%%%%%%%%%%%%%%%%%%%%%%%%%%%%%%%%%%%%%%%%%%%%%%%

\begin{abstract} 
The main objective of this paper is twofold. One is to classify and construct 
$SL(3,\R)$-intertwining differential operators between vector bundles
over the real projective space $\RP^2$. 
It turns out that two kinds of operators appear. 
We call them Cartan operators and PRV operators.
The second objective is then to study the representations realized 
on the kernel of those operators both in the smooth and holomorphic setting.
A key machinery is the BGG resolution. 
In particular, by exploiting some results 
of Davidson--Enright--Stanke and Enright--Joseph, 
the irreducible unitary highest 
weight modules of $SU(1,2)$ at the (first) reduction points are classified by 
the image of Cartan operators and kernel of PRV operators.
\end{abstract}

%%%%%%%%%%%%%%%%%%%%%%%%%%%%%%%%%%%%%%%%%%
%\setcounter{tocdepth}{1}
%\tableofcontents

%%%%%%%%%%%%%%%%%%%%%%%%%%%%%%%%%%%%%%%%%%%%%%%
\section{Introduction}\label{sec:intro}

%%%%%%%%%%%%%%%%%%%%%%%%%%%%%%%%%%%%%%%%%%%%%%%

In this paper we study from a modern point of view the very classical topic of real projective geometry in dimension
two;  this was already studied in Greek mathematics, and in the Renaissance 
it became part of the theory of
perspective in art. Here there is (from the point of view of Sophus Lie) the 
8-dimensional group $G=SL(3,\R)$ of projective transformations
acting on $G/P$ with (in modern language) $P$ a parabolic subgroup. As such it falls into the class of parabolic geometries, and
also the kind of structure that Joseph A.\ Wolf studied in depth, namely flag manifolds. Our aim in this paper is to illustrate
in detail several aspects of this geometry in finding some invariants associated with homogeneous vector bundles
on $G/P$; these are $G$-intertwining differential operators for the natural action on sections, and we shall see how they fit
into the very general theory of parabolic geometry. This involves by a natural duality of Verma modules and the celebrated BGG resolution, and we will see natural constructions of representations of $G$ in spaces of sections.
It turns out that it is related to the classification of some irreducible unitary highest weight representations of another real form $G_u=SU(1,2)$ of $G_\C=SL(3,\C)$.
The method uses some aspects of the theory of the Rankin--Cohen operators, known from the theory of automorphic forms;
this allows quite explicit formulas for the differential operators in question; as such we feel that the explicit case of 2-dimensional
projective geometry deserves attention.
 In the following, we shall describe our work in more detail.

%%%%%%%%%%%%%%%%%%%%%%%%%%%%%%%%%%%%%%%%%%%%%%%
\subsection{Classification and construction}
First, we briefly introduce some notation;
more detailed notation will be introduced in Section \ref{sec:SL3}.
Let $G$ be a real reductive group and $P=MAN_+$ a Langlands 
decomposition of a parabolic subgroup $P$.
We write $\Irr(M)_\fin$ for the set of  equivalence classes of 
irreducible finite-dimensional representations of $M$.
Likewise, let $\Irr(A)$ denote the set of characters of $A$. 
Then, for the outer tensor product $\varpi\boxtimes \nu \boxtimes \triv$ 
of $\varpi \in \Irr(M)_\fin$, $\nu \in \Irr(A)$, and the trivial representation
$\triv$ of $N_+$, we put
\begin{equation*}
I(\varpi,\nu) := \Ind_{P}^G(\varpi\boxtimes \nu \boxtimes \triv)
\end{equation*}
for an unnormalized parabolically induced representation of $G$.

Let $\Diff_G(I(\sigma,\lambda), I(\varpi,\nu))$ denote the space of $G$-intertwining
differential operators $\D \colon I(\sigma,\lambda) \to I(\varpi,\nu)$. In what follows,
we often refer to them as IDOs.

In this paper, we consider the following problem for the case
$(G, P)=(SL(3,\R), P_{1,2})$, where $P_{1,2}$ is a maximal parabolic subgroup 
of $G$ with partition $3=1+2$.

 \begin{prob}\label{prob:A}
Do the following.
\begin{enumerate}
\item[\emph{(A1)}] Classify $(\sigma, \lambda), (\varpi, \nu) \in \Irr(M)_\fin\times \Irr(A)$ 
such that
\begin{equation*}
\Diff_G(I(\sigma,\lambda), I(\varpi,\nu))\neq\{0\}.
\end{equation*}
%\vskip 0.1in

\item[\emph{(A2)}] 
For $(\sigma, \lambda), (\varpi, \nu) \in \Irr(M)_\fin\times \Irr(A)$,
determine the dimension 
\begin{equation*}
\dim \Diff_G(I(\sigma,\lambda), I(\varpi,\nu)).
\end{equation*}

\item[\emph{(A3)}] 
For $(\sigma, \lambda), (\varpi, \nu) \in \Irr(M)_\fin\times \Irr(A)$,
construct generators 
\begin{equation*}
\D \in \Diff_G(I(\sigma,\lambda), I(\varpi,\nu)).
\end{equation*}
\end{enumerate}
\end{prob}

For recent developments on Problem \ref{prob:A}, see, for instance,
the introduction of \cite{KuOr25a} and references therein. 
The most outstanding recent advance in Problem \ref{prob:A} is the invention 
of the so-called F-method \cite{KP1, KP2}. 
This is a powerful machinery proposed by T.\ Kobayashi,
which reformulates the problems (A1), (A2), and (A3) to solving a certain system
of partial differential equations. As such, the F-method allows one to 
achieve the construction and classification of IDOs, simultaneously. 
The method plays a pivotal role especially in taking care of the differential symmetry breaking operators (DSBOs), 
which are $G'$-IDOs for $G' \subset G$ from a $G$-representation
to a $G'$-representation. For instance, in \cite{Kubo24+}, the first author tackles Problem \ref{prob:A} for DSBOs $\D$  by using the F-method.

It is no harm to apply the machinery to the present situation.
In fact, it is used in our earlier work \cite{KuOr25a} to solve Problem \ref{prob:A} for the case $(G, P)=(SL(n,\R), P_{1,n-1})$ with $\dim \sigma =1$ (line bundle case).
However, in this paper,  we take a more classical approach to 
the problem, which stems from 
the duality between the space of IDOs and 
that of homomorphisms between generalized Verma modules,
as it appears simpler in the present case. Combining the duality 
with the idea of ``standard map'' of generalized Verma modules, we shall solve
(A1) and (A2) in Section \ref{sec:classification}.
It turned out that,
for any $m, \ell \in \Z_{\geq 0}$, there exists a non-zero IDO $\D$  
with order $|\ell-m|$. See Theorems \ref{thm:IDO1} and \ref{thm:IDO2}
for more details.

To solve (A3), one needs to compute $P$-invariants of generalized Verma modules.
In doing so, there are, in principle, two steps to consider.
The first is to compute $MA$-invariants of the tensor product of three
irreducible finite-dimensional representations of $MA$.
(On this matter, see \cite[Thms.\ 4.3 and 5.5]{DES90} and \cite[Thm.\ 3]{EG11} 
as relevant works). 
The second is the evaluation of the $\fn_+$-invariance of the $MA$-invariants,
where $\fn_+$ is the complexified Lie algebra of $N_+$.
This step can be thought of as computing 
a \emph{conformal weight} in conformal geometry.

The first step for the present work can be reduced to computing the projection 
of the tensor product of two irreducible finite-dimensional representations 
of $SL(2,\R)$ onto its irreducible components. Such projections are known as 
Gordan's transvectants, or equivalently, Rankin--Cohen brackets
(see, for instance, \cite{Clerc25, ElGradechi06} for a relationship between them).
For this purpose, one can make a use of the explicit formula of
Molchanov \cite{Molchanov15}. We, nonetheless, compute the $MA$-invariance
in a more elementary way (see Section \ref{sec:proof2}).

The evaluation of the $\fn_+$-invariance in the second step 
is in general not easy, especially 
for a vector bundle case.
In the present work, we succeeded to 
avoid such a computation by showing that any IDO in the present case
comes from the standard map between generalized Verma modules.
This is done in Section \ref{sec:Step4}.

It is noteworthy that only the so-called Cartan component and PRV
component among all $SL(2,\R)$-irreducible components contribute to 
the construction of IDOs. This verifies some insight made by Kable in 
\cite[Intro.]{Kable18A} on his \emph{automatic conformal invariance} result 
for second order systems on vector bundles. In this paper, we refer to
the IDOs $\D$ coming from the Cartan component and PRV component as
the \emph{Cartan operator} and \emph{PRV operator}, respectively.

\vskip 0.1in

%%%%%%%%%%%%%%%%%%%%%%%%%%%%%%%%%%%%%%%%%%%%%%%
\subsection{BGG resolution}
The kernel $\Ker \D$ and image $\Im \D$ 
of IDOs $\D$ are naturally representations of $G$.
These representations are closely related to the composition structure of the 
parabolically induced representations. Indeed, in \cite{KuOr25a},
we exploited the composition structure studied by Howe--Lee \cite{HL99} and 
van Dijk--Molchanov \cite{vDM99} to investigate the representations on
$\Ker\D$ and $\Im \D$ for the case $(G, P)=(SL(n,\R), P_{1,n-1})$ 
with $\dim \sigma =1$. (See Remark \ref{rem:Ktype} for more details.)

As opposed to the line bundle case, the composition structure for a vector 
bundle case seems not fully investigated. Then, in the present work, we study
$\Ker \D$ and $\Im \D$ via the BGG resolution. A BGG resolution, or more 
generally, BGG sequence (\cite{CD01, CSS01}) is a powerful tool to construct
IDOs $\D$, which are also referred to as 
\emph{equivariant differential operators},
\emph{invariant differential operators}, or \emph{conformally covariant differential operators} (in conformal geometry), among others.
Eastwood--Gover \cite{EG11} carefully studied BGG resolutions 
for the case $(G, P)=(SL(n,\R), P_{1,n-1})$. With the aid of their results,
we shall show that the Cartan operators and PRV operators are the first and second 
BGG operators, respectively. This in particular shows that every irreducible 
finite-dimensional representation of $G$ is realized in the kernel $\Ker \D$
with a suitable parity condition.

It is remarked that, as $M$ for $P=MAN_+$ 
is $M\simeq SL^{\pm}(2,\R)$,
the representations $\sigma\in\Irr(M)_{\fin}$ involve some parity.
BGG sequences are well studied in parabolic geometry; nonetheless,
such a condition seems not paid attention
as it is automatically 
fixed by the irreducible finite-dimensional representation realized in $\Ker \D$. 
The  IDOs $\D$ are in fact ``rigid'' with respect to the parity of $\sigma \in \Irr(M)_\fin$
(cf.\ \cite[Lem.\ 2.17]{KuOr19}).
So, if the BGG resolution is utilized to determine the representations 
in $\Ker \D$ and $\Im \D$, then
one should be careful for the parity of representations of $M$.
We shall return and discuss this point in more detail in Section \ref{sec:main6}
(especially after Corollary \ref{cor:BGGres}).

%%%%%%%%%%%%%%%%%%%%%%%%%%%%%%%%%%%%%%%%%%%%%%%
\subsection{Irreducible unitary highest weight modules}
As there exists a duality between the space of IDOs $\D$ and that of
homomorphisms between
generalized Verma modules, the IDOs $\D$ are determined by the pair of 
complex Lie algebras $(\fg, \fp)$, where $\fg$ and $\fp$ are the complexified
Lie algebras of $G=SL(3,\R)$ and 
the parabolic subgroup $P$. So, the operators $\D$ 
are intertwining operators also for another real form $G_u=SU(1,2)$.

More precisely speaking,
let $K$ be a maximal compact subgroup of $G_u$. Then, via the Borel embedding
$G_u/K \hookrightarrow G_\C/P_\C$, where $P_\C$ is the complexification
of a parabolic subgroup $P$, 
IDOs $\D$ are regarded as $G_u$-intertwining differential operators on
the space of holomorphic sections over $G_u/K$. In this context, 
IDOs $\D$ are more often referred to as \emph{covariant differential operators}
(see, for instance, \cite{HJ82, HJ83, Jakobsen85}).

In the above setting with $G_\C/\bar{P}_\C$, where $\bar{P}_\C$ is the 
complexification of the opposite parabolic subgroup $\bar{P}$ to $P$,
Davidson--Enright--Stanke \cite{DES90, DES91} uses 
IDOs $\D$ to classify unitary highest weight modules at the so-called reduction points.
In their study, the aforementioned PRV components play a central role. 
(For a connection between unitary highest 
weight modules and PRV components, see also Enright--Joseph \cite{EJ90},
Enright--Wallach \cite{EW97},
Pand{\v z}i{\' c}--Prli{\' c}--Sou{\v c}ek--Tu{\v c}ek \cite{PPST23+},
and Pand{\v z}i{\' c}--Prli{\' c}--Savin--Sou{\v c}ek--Tu{\v c}ek \cite{PPSST24+},
among others.)

There is only one reduction point (first reduction point) for $G_u=SU(1,2)$.
Then, at the end of the paper, we shall classify those unitary representations
by the kernel of PRV operators and image of Cartan operators
via the BGG resolution. This is done in Theorem \ref{thm:UHW}.

%%%%%%%%%%%%%%%%%%%%%%%%%%%%%%%%%%%%%%%%%%%%%%%
\subsection{Organization of the paper}
Now we describe the rest of the paper. There are seven sections including this 
introduction. In Section \ref{sec:Prelim}, we recall the duality between
the space of IDOs $\D$ and that of homomorphisms between
generalized Verma modules. We also 
review quickly the notion of standard map between generalized Verma modules.
Then, in Section \ref{sec:SL3}, we specialize the framework considered 
in Section \ref{sec:Prelim} to $(G,P)= (SL(3,\R), P_{1,2})$. In particular, we 
fix some notation for $\Irr(M)_{\fin}$ and $\Irr(A)$, carefully.

Sections \ref{sec:classification} and \ref{sec:construction} are devoted to
the classification of parameters ((A1) and (A2))
and explicit construction of IDOs $\D$ ((A3)), respectively. 
The main results, namely, Theorems \ref{thm:IDO1} and \ref{thm:IDO2}
solve Problem \ref{prob:A} for the present case. In addition,
we also study the algebraic counterpart of Problem \ref{prob:A} 
for generalized Verma modules. These are achieved in 
Theorems \ref{thm:gP1} and \ref{thm:gP2}.

The representations in the kernel $\Ker \D$ and image $\Im \D$ are studied in 
Section \ref{sec:BGG}. 
This is done via the BGG resolution in Corollary \ref{cor:BGGres}.
In this section we carefully investigate the effect of the parity conditions
to $\Ker \D$ and image $\Im \D$.

The last section, Section \ref{sec:SU}, is for the classification of  the irreducible 
unitary highest weight modules of $SU(1,2)$  at the (first) reduction points.
We first quickly review a general theory of the Borel embedding 
$G_u/K \hookrightarrow G_\C/\bar{P}_\C$. Then we give a classification
of such representations in terms of the kernel of PRV operators 
and image of Cartan operators via the BGG resolution. 
This is achieved in Theorem \ref{thm:UHW}.

%%%%%%%%%%%%%%%%%%%%%%%%%%%%%%%%%%%%%%%%%%%%%%%
\section{Preliminaries}\label{sec:Prelim}

The aim of this section is to recall from the literature 
the so-called duality theorem
(Theorem \ref{thm:duality}) and the notion of standard map
between generalized Verma modules.
Many of the material in this section is taken from \cite{KuOr25a}.
In later sections, we shall apply those ideas to the case $(G,P)=(SL(3,\R), P_{1,2})$.

%%%%%%%%%%%%%%%%%%%%%%%%%%%%%%%%%%%%%%%%%%
\subsection{Duality theorem}\label{sec:prelim}
Let $G$ be a real reductive Lie group and $P=MAN_+ $ a Langlands decomposition of a parabolic subgroup $P$ of $G$. We denote by $\fg(\R)$ and 
$\fp(\R) = \fm(\R) \oplus \fa(\R) \oplus \fn_+(\R)$ the Lie algebras of $G$ and 
$P=MAN_+$, respectively.

 For a real Lie algebra $\f{y}(\R)$, we write $\f{y}$
and $\Cal{U}(\fy)$ for its complexification and the universal enveloping algebra of 
$\fy$, respectively. 
For instance, $\fg, \fp, \fm, \fa$, and $\fn_+$ are the complexifications of $\fg(\R), \fp(\R), \fm(\R), \fa(\R)$, and $\fn_+(\R)$, 
respectively.

For $\lambda \in \fa^* \simeq \Hom_\R(\fa(\R),\C)$,
we denote by $\C_\lambda$ 
the one-dimensional representation of $A$ defined by 
$a\mapsto a^\lambda:=e^{\lambda(\log a)}$. 
For a finite-dimensional 
representation $(\sigma, V)$ of $M$ and $\lambda \in \fa^*$,
we denote by $\sigma_\lambda$ the outer tensor product representation $\sigma \boxtimes \C_\lambda$. As a representation on $V$, we define $\sigma_\lambda \colon
ma \mapsto a^\lambda\sigma(m)$. By letting $N_+$ act trivially, we regard 
$\sigma_\lambda$ as a representation of $P$. Let $\Cal{V}:=G \times_P V \to G/P$
be the $G$-equivariant vector bundle over the real flag variety $G/P$
 associated with the 
representation $(\sigma_\lambda, V)$ of $P$. We identify the Fr{\' e}chet space 
$C^\infty(G/P, \Cal{V})$ of smooth sections with 
\begin{equation*}
C^\infty(G, V)^P:=\{f \in C^\infty(G,V) : 
f(gp) = \sigma_\lambda^{-1}(p)f(g)
\;\;
\text{for $p \in P$}\},
\end{equation*}
the space of $P$-invariant smooth functions on $G$.
Then, via the left regular representation $L$ of $G$ on $C^\infty(G)$,
we realize the parabolically induced representation 
$\pi_{(\sigma, \lambda)} = \Ind_{P}^G(\sigma_\lambda)$ on $C^\infty(G/P, \Cal{V})$.
We denote by $R$ the right regular representation of $G$ on $C^\infty(G)$.

Similarly, for a finite-dimensional 
representation $(\eta_\nu, W)$ of $MA$, we define
the induced representation $\pi_{(\eta, \nu)}=\Ind_P^G(\eta_\nu)$ on the 
space $C^\infty(G/P, \Cal{W})$ of smooth sections for a $G$-equivariant vector bundle 
$\Cal{W}:=G\times_PW \to G/P$.
We write $\Diff_G(\Cal{V},\Cal{W})$ for the space of intertwining differential operators
$\D \colon C^\infty(G/P, \Cal{V}) \to C^\infty(G/P, \Cal{W})$.

Let $\mathfrak{g}(\R)=\mathfrak{n}_-(\R) \oplus \mathfrak{m}(\R) 
\oplus \mathfrak{a}(\R) \oplus \mathfrak{n}_+(\R)$ be the 
Gelfand--Naimark decomposition of $\mathfrak{g}(\R)$,
and write $N_- = \exp(\fn_-(\R))$. We identify $N_-$ with the 
open Bruhat cell $N_-P$ of $G/P$ via the embedding 
$\iota\colon N_- \hookrightarrow G/P$, $\bar{n} \mapsto \bar{n}P$.
Via the restriction of the vector bundle $\Cal{V} \to G/P$ to the open Bruhat cell
$N_-\stackrel{\iota}{\hookrightarrow} G/P$,
we regard $C^\infty(G/P,\mathcal{V})$ as a subspace of 
$C^\infty(N_-) \otimes V$.

We view intertwining differential operators 
$\D \colon
C^\infty(G/P,\mathcal{V})
\to C^\infty(G/P,\mathcal{W})$
as differential operators
$\D' \colon C^\infty(N_-) \otimes V
\to C^\infty(N_-) \otimes W$ such that
the restriction $\D'\vert_{C^\infty(G/P,\mathcal{V})}$
to $C^\infty(G/P,\mathcal{V})$ is a map
$\D'\vert_{C^\infty(G/P,\mathcal{V})}\colon
C^\infty(G/P,\mathcal{V})
\to C^\infty(G/P,\mathcal{W})$ (see \eqref{eqn:21} below).
\begin{equation}\label{eqn:21}
\xymatrix{
C^\infty(N_-) \otimes V 
\ar[r]^{\mathcal{D}' } 
& C^\infty(N_-) \otimes W\\ 
C^\infty(G/P,\mathcal{V}) 
\;\; \ar[r]_{\stackrel{\phantom{a}}{\hspace{20pt}\mathcal{D}=\mathcal{D}' \vert_{\small{C^\infty(G/P,\mathcal{V})}}} }
 \ar@{^{(}->}[u]^{\iota^*}
& \;\; C^\infty(G/P,\mathcal{W}) \ar@{^{(}->}[u]_{\iota^*}
}
\end{equation}

\noindent
In particular, we regard $\Diff_G(\Cal{V},\Cal{W})$ as 
\begin{align}
\Diff_G(\Cal{V},\Cal{W}) 
&\subset \Diff_\C(C^\infty(N_-)\otimes V, C^\infty(N_-)\otimes W) \label{eqn:DN}\\[3pt]
&\simeq \Diff_\C(C^\infty(N_-)) \otimes \Hom_\C(V, W) \nonumber\\[3pt]
&\simeq \Diff_\C(C^\infty(N_-)) \otimes V^\vee \otimes  W,\label{eqn:DN3}
\end{align}
where $\Diff_\C(C^\infty(N_-))$ denotes the space of differential operators on $C^\infty(N_-)$
and $V^\vee := \Hom_\C(V,\C)$.

For a finite-dimensional representation $(\sigma_\lambda,V)$ of $MA$, 
we write $((\sigma_\lambda)^\vee, V^\vee)$ for
the contragredient representation of $(\sigma_\lambda,V)$. By letting $\fn_+$ act 
on $V^\vee$ trivially, we regard the infinitesimal representation 
$d\sigma^\vee \boxtimes \C_{-\lambda}$ of $(\sigma_\lambda)^\vee$ as a $\fp$-module. The induced module
\begin{equation*}
M_\fp(V^\vee) := \Cal{U}(\fg)\otimes_{\Cal{U}(\fp)}V^\vee
\end{equation*}
is called a generalized Verma module, provided that $V$ is a simple $\fp$-module.
Via the diagonal action of $P$ on $M_\fp(V^\vee)$, we regard
$M_\fp(V^\vee)$ as a $(\fg, P)$-module.
For the proof of the following theorem, see, for instance,
\cite{CS90, KP1, KR00}.

\begin{thm}[duality theorem]\label{thm:duality}
There is a natural linear isomorphism
\begin{equation}\label{eqn:duality1}
\EuScript{D}
\colon
\operatorname{Hom}_{P}(W^\vee,M_{\fp}(V^\vee))
\stackrel{\sim}{\To} 
\operatorname{Diff}_{G}(\mathcal V, \mathcal W),
\end{equation}
where,
for $\varphi \in \Hom_P(W^\vee, M_\fp(V^\vee))$ and
$F \in C^\infty(G/P,\Cal{V})\simeq C^\infty(G,V)^P$,
the element $\EuScript{D}(\varphi)F \in C^\infty(G/P,\Cal{W})\simeq C^\infty(G,W)^P$
is given by 
\begin{equation}\label{eqn:HD}
\IP{\EuScript{D}(\varphi)F}{w^\vee} =
\sum_{j}\IP{dR(u_j)F}{v_j^\vee}
\;\; \text{for $w^\vee \in W^\vee$},
\end{equation}
where $\varphi(w^\vee)=\sum_j u_j\otimes v_j^\vee \in 
M_\fp(V^\vee)$.
\end{thm}

\begin{rem}
By the Frobenius reciprocity, the isomorphism \eqref{eqn:duality1} is equivalent to
\begin{equation}\label{eqn:duality2}
\EuScript{D}
\colon
\operatorname{Hom}_{\fg, P}(M_\fp(W^\vee),M_{\fp}(V^\vee))
\stackrel{\sim}{\To} 
\operatorname{Diff}_{G}(\mathcal V, \mathcal W).
\end{equation}

\end{rem}

%%%%%%%%%%%%%%%%%%%%%%%%%%%%%%%%%%%%%%%%%%
\subsection{Standard map}\label{sec:Std1}
To introduce the notion of standard map between
generalized Verma modules, we first reparametrize
them in terms of infinitesimal characters.

Let $\fg$ be a complex simple Lie algebra. Fix a Cartan subalgebra $\fh$ and 
write $\gD\equiv \gD(\fg, \fh)$ for the set of roots of $\fg$ with respect to $\fh$.
Choose a positive system $\gD^+$ and denote by $\Pi$ the set of simple roots of 
$\gD$. Let $\fb$ denote the Borel subalgebra of $\fg$ associated with $\gD^+$, namely,
$\fb = \fh \oplus \bigoplus_{\ga \in \gD^+} \fg_\ga$, where $\fg_\ga$ is the root space
for $\ga \in \gD^+$.

Let $\IP{\cdot}{\cdot}$ denote the inner product on $\fh^*$ induced from 
a non-degenerate symmetric bilinear form 
on $\fg$. For $\ga \in \gD$, we write $\ga^\vee = 2\ga/\IP{\ga}{\ga}$. Also, write 
$s_\ga$ for the root reflection with respect to $\ga \in \gD$. As usual, we let 
$\rho=(1/2)\sum_{\ga \in \gD^+}\ga$ be half the sum of the positive roots.

Let $\fp \supset \fb$ be a standard parabolic subalgebra of $\fg$. Write 
$\fp =\fl \oplus \fn_+$ for the Levi decomposition of $\fp$. We let 
$\Pi(\fl) = \{\ga \in \Pi : \fg_\ga \subset \fl\}$.

Now we put
\begin{equation}\label{eqn:Pell}
\mathbf{P}^+_{\fl}:=\{\mu \in \fh^* : \IP{\mu}{\ga^\vee} \in 1+\Z_{\geq 0}
\;\;
\text{for all $\ga \in \Pi(\fl)$}\}.
\end{equation}
For $\mu \in \mathbf{P}^+_\fl$, let $E(\mu-\rho)$ be the finite-dimensional 
simple $\Cal{U}(\fl)$-module with highest weight $\mu-\rho$. 
By letting $\fn_+$ act trivially, we regard $E(\mu-\rho)$ as a $\Cal{U}(\fp)$-module.
Then the induced module
\begin{equation}\label{eqn:Verma2}
N_\fp(\mu):=\Cal{U}(\fg)\otimes_{\Cal{U}(\fp)} E(\mu-\rho)
\end{equation}
is the generalized Verma module with highest weight $\mu-\rho$.
If $\fp = \fb$, then $N(\mu) \equiv N_{\fb}(\mu)$ is the (ordinary) Verma module with highest weight $\mu-\rho$.  

\vskip 0.1in

%%%%%%%%%%%%%%%%%%%%%%%%%%%%%%%%%%%%%%%%%%
%\subsection{Standard map}\label{sec:Std2}

Let $\mu, \eta \in \mathbf{P}^+_\fl$.
It follows from a theorem by BGG--Verma 
(Theorem \ref{thm:link} below)
that if $\Hom_{\fg}(N_\fp(\mu), N_\fp(\eta)) \neq \{0\}$, then 
$\Hom_{\fg}(N(\mu), N(\eta)) \neq \{0\}$
(cf.\ \cite[Thm.\ 7.6.23]{Dix96} and \cite[Thm.\ 5.1]{Hum08}).

Conversely,
suppose that there exists a non-zero
$\fg$-homomorphism $\varphi\colon N(\mu)\to N(\eta)$.
Let $\pr_\mu\colon N(\mu) \to N_\fp(\mu)$ denote 
the canonical projection map. Then we have 
$\varphi(\Ker(\pr_\mu)) \subset \Ker(\pr_\eta)$
(\cite[Prop.\ 3.1]{Lepowsky77}). Thus, the map $\varphi$ induces a 
$\fg$-homomorphism $\varphi_{\std} \colon N_\fp(\mu)\to N_\fp(\eta)$ such that
the following diagram commutes.
\begin{equation*}
\xymatrix@R-=0.8pc@C-=0.5cm{
N(\mu) \ar[rr]^\varphi \ar[dd]_{\pr_\mu} && N(\eta \ar[dd]^{\pr_\eta})\\
&\circlearrowleft&\\
N_\fp(\mu)\ar[rr]^{\varphi_{\std}}& & N_\fp(\eta)
}
\end{equation*}

The map $\varphi_{\std}$ is called the \emph{standard map} from $N_\fp(\mu)$ to 
$N_\fp(\eta)$ associated with $\varphi$ (\cite[p.\ 501]{Lepowsky77}). As $\dim \Hom_{\fg}(N(\mu), N(\eta)) \leq 1$, 
the standard map $\varphi_{\std}$ is unique up to scalar. 
It is known that the standard map $\varphi_{\std}$ could be zero, and even if 
$\varphi_{\std}=0$, there could be another non-zero map from $N_\fp(\mu)$ to 
$N_\fp(\eta)$. Any homomorphism that is not standard is called 
a \emph{non-standard map}.

It is known when the standard map $\varphi_\std$ is zero.
To state the criterion, we introduce the notion of a link between two weights.

\begin{defn}[{Bernstein--Gelfand--Gelfand}]\label{def:link}
Let $\mu, \eta \in \fh^*$ and $\beta_1,\ldots, \beta_t \in \gD^+$. Set
$\eta_0 := \eta$ and $\eta_i := s_{\beta_i}\cdots s_{\beta_1}\eta$ for $1\leq i \leq t$.
We say that the sequence $(\beta_1,\ldots, \beta_t)$ links $\eta$ to $\mu$ if 
it satisfies the following two conditions.
\begin{enumerate}[(1)]
\item $\eta_t = \mu$.
\vskip 0.05in
\item $\IP{\eta_{i-1}}{\beta^\vee_{i}} \in \Z_{\geq 0}$ 
for all $i \in \{1,\ldots, t\}$.
\end{enumerate} 
\end{defn}

Let $L(\mu)$ denote the unique irreducible quotient of the Verma module
$N(\mu)$.

\begin{thm}[{BGG--Verma}]\label{thm:link}
The following conditions on $(\mu, \eta) \in (\fh^*)^2$ are equivalent.

\begin{enumerate}
\item[\emph{(i)}] $\Hom_\fg(N(\mu), N(\eta)) \neq \{0\}$.
\item[\emph{(ii)}] $L(\mu)$ is a composition factor of $N(\eta)$.
\item[\emph{(iii)}] There exists a sequence $(\beta_1, \ldots, \beta_t)$ with 
$\beta_t \in \gD^+$ that links $\eta$ to $\mu$.
\end{enumerate}
\end{thm}

The criterion on the vanishing of the standard map $\varphi_{\std}$ is first
studied by Lepowsky (\cite{Lepowsky77}) and then Boe refined Lepowsky's 
criterion (\cite{Boe85}). 
The next theorem is a version of Boe's criterion
\cite[Thm.\ 3.3]{Boe85}.

\begin{thm}\label{thm:Boe}
Let $\mu, \eta \in \fh^*$ and suppose that there exists a non-zero 
$\fg$-homomorphism
$\varphi \in \Hom_{\fg}(N(\mu), N(\eta))$.
Then the following two conditions on $(\mu,\eta)$ are equivalent.
\begin{enumerate}
\item[\emph{(i)}]
The standard map $\varphi_{\std}\colon N_\fp(\mu) \to N_\fp(\eta)$ 
associated with $\varphi$ is non-zero.
\item[\emph{(ii)}]
For all sequences 
$(\beta_1, \ldots, \beta_t)$ linking $\eta$ to $\mu$,
we have $\eta_1 \in \mathbf{P}^+_\fl$.
\end{enumerate}
\end{thm}

%%%%%%%%%%%%%%%%%%%%%%%%%%%%%%%%%%%%%%%%%%
\section{Specialization to $(SL(3,\R), P_{1,2})$}
\label{sec:SL3}

The aim of this short section is to introduce necessary notation for the rest of the 
paper. The expositions of this section are mainly a specialization of
\cite{KuOr25a} to $n=3$.

%%%%%%%%%%%%%%%%%%%%%%%%%%%%%%%%%%%%%%%%%%
\subsection{Notation}\label{sec:notation}
Let $G = SL(3,\R)$ with Lie algebra $\fg(\R)=\f{sl}(3,\R)$.
We put
\begin{equation}\label{eqn:Npm}
N_j^+:=E_{1,j+1},
\quad
N_j^-:=E_{j+1,1}
\quad
\text{for $j\in \{1, 2\}$}
\end{equation}
and 
\begin{equation*}
H_0:=\frac{1}{3}(2E_{1,1} - E_{2,2}- E_{3,3})
=\frac{1}{3}\diag(2, -1, -1),
\end{equation*}
where $E_{i,j}$ denote the matrix units. We normalize $H_0$ as
$\wH_0:=\frac{3}{2}H_0$,
namely,
\begin{equation}\label{eqn:H0}
\wH_0=\frac{1}{2}(2E_{1,1} -E_{2,2}-E_{3,3})
=\frac{1}{2}\diag(2, -1, -1).
\end{equation}

Let
\begin{align}
\fn_+(\R) =\Ker(\ad(H_0)-\id) 
&= \Ker(\ad(\wH_0)-\tfrac{3}{2}\id),\label{eqn:nR1}\\[3pt]
\fn_-(\R)=\Ker(\ad(H_0)+\id) 
&=\Ker(\ad(\wH_0)+\tfrac{3}{2}\id).\label{eqn:nR}
\end{align}
Then we have
\begin{equation*}
\fn_{\pm}(\R)=\spn_{\R}\{N_1^{\pm}, N_{2}^{\pm} \}.
\end{equation*}

For $X, Y \in \fg(\R)$, let $\Tr(X,Y)=\text{Trace}(XY)$ denote the trace form of $\fg(\R)$. 
Then $N_i^+$ and $N_j^-$ satisfy $\Tr(N_i^+,N_j^-)=\delta_{i,j}$.
In what follows, we identify  the dual $\fn_-(\R)^\vee$
with $\fn_-(\R)^\vee \simeq \fn_+(\R)$ via the trace form $\Tr(\cdot, \cdot)$.

We put $\fa(\R):= \R \wH_0$ and
\begin{equation}\label{eqn:Lm}
\fm(\R):= \left\{
\begin{pmatrix}
0 & \\
 & X
\end{pmatrix}
:
X \in \f{sl}(2,\R)
\right\}\simeq \f{sl}(2,\R).
\end{equation}
We have $\fm(\R)\oplus \fa(\R) = \Ker(\ad(\wH_0))$ and
the decomposition
$\fg(\R) = \fn_-(\R) \oplus \fm(\R) \oplus \fa(\R) \oplus \fn_+(\R)$
is a Gelfand--Naimark decomposition of $\fg(\R)$.
The subalgebra
$\fp(\R):=\fm(\R) \oplus \fa(\R) \oplus \fn_+(\R)$ is a maximal parabolic 
subalgebra of $\fg(\R)$ and the nilpotent radicals $\fn_{\pm}(\R)$ are abelian.

Let $P=N_G(\fp(\R))$, the normalizer of $\fp(\R)$ in $G$. 
We write $P=MAN_+$ for the Langlands decomposition of $P$ corresponding to 
$\fp(\R)=\fm(\R) \oplus \fa(\R) \oplus \fn_+(\R)$. 
Then $A = \exp(\fa(\R)) = \exp(\R \wH_0)$ and 
$N_+=\exp(\fn_+(\R))$. The group $M$ is given by
\begin{equation*}
M= 
\left\{
\begin{pmatrix}
\det(g)^{-1} &\\
& g\\
\end{pmatrix}
:
g \in SL^{\pm}(2,\R)
\right\}\simeq SL^{\pm}(2,\R).
\end{equation*}

For a closed subgroup $J$ of $G$, we denote by $\Irr(J)$ and $\Irr(J)_{\fin}$
the sets of equivalence classes of irreducible representations of $J$  and 
irreducible finite-dimensional representations of $J$, respectively.

For $\lambda \in \C$, 
we define a one-dimensional representation 
$\C_\lambda=(\chi^\lambda, \C)$ of $A=\exp(\R\wH_0)$ by
\begin{equation}\label{eqn:chi}
\chi^\lambda \colon \exp(t \wH_0) \longmapsto \exp(\lambda t).
\end{equation}
Then $\Irr(A)$ is given by
\begin{equation*}
\Irr(A)=\{\C_\lambda : \lambda \in \C\} \simeq \C.
\end{equation*}

For $\alpha \in \{\pm\}\simeq \Z/2\Z$, 
a one-dimensional representation $\C_\alpha$ of $M$ 
is defined  by
\begin{equation*}
\begin{pmatrix}
\det(g)^{-1} &\\
& g\\
\end{pmatrix}
\longmapsto
\sgn^\alpha(\det(g)),
\end{equation*}
where
\begin{equation*}
\sgn^\alpha(\det(g)) 
= 
\begin{cases}
1 & \text{if $\alpha=+$},\\
\sgn(\det(g)) & \text{if $\alpha=-$}.
\end{cases}
\end{equation*}

To describe $\Irr(M)_{\fin}$,
let $S^m(\C^2)$ denote the space of symmetric tensors on $\C^2$ of homogeneous 
degree $m$. We then write $\Pol^m(\C^2) := S^m((\C^2)^\vee)$,  the space of polynomial functions on $\C^2$ of homogeneous degree $m$.
Since
\begin{equation}\label{eqn:dualsym}
\Pol^m(\C^2) = S^m((\C^2)^\vee)\simeq S^m(\C^2)^\vee, 
\end{equation}
the two spaces $S^m(\C^2)$ and $\Pol^m(\C^2)$ are dual to each other.

Let 
$\sym_2^m$ and $\poly_2^m$ denote 
the irreducible representations of $SL(2,\R)$  on $S^m(\C^2)$ and $\Pol^m(\C^2)$,
respectively. By the representation theory of $\f{sl}(2,\C)$, we have
\begin{align*}
\Irr(SL(2,\R))_\fin 
&= \{(\sym_2^m, S^m(\C^2)): m \in  \Z_{\geq 0}\}\\
&= \{(\poly_2^m, \Pol^m(\C^2)): m \in  \Z_{\geq 0}\}.
\end{align*}
Then $\Irr(M)_\fin = \Irr(SL^{\pm}(2,\R))_\fin$ is given by
\begin{align*}
\Irr(M)_{\fin}
&=
\{\C_\alpha \otimes \varpi: 
(\alpha, \varpi) \in \{\pm\} \times \Irr(SL(2,\R))_{\fin}\}\\
&
\simeq
\Z/2\Z \times \Z_{\geq 0}.
\end{align*}
Since $\Irr(P)_{\fin}\simeq \Irr(M)_{\fin} \times \Irr(A)$, the set $\Irr(P)_\fin$ can be 
parametrized by
\begin{equation*}
\Irr(P)_{\fin}\simeq 
\Z/2\Z \times \Z_{\geq 0} \times \C.
\end{equation*}

For $(\ga, \poly_2^m, \lambda) \in
\Irr(P)_{\fin}$,
we write
\begin{equation}\label{eqn:Ind}
I(\poly_2^m,\lambda)^\ga 
:= 
\Ind_{P}^G\left((\C_\ga \otimes \poly_2^m)\boxtimes \C_\lambda\right)
\end{equation}
for the unnormalized parabolically induced representation 
$\Ind_{P}^G\left((\C_\ga \otimes \poly_2^m)\boxtimes \C_\lambda \right)$ 
of $G$.

Likewise, for $(\ga, \sym_2^m, \lambda) \in \Irr(P)_{\fin}$,
we write 
\begin{equation}\label{eqn:Verma}
M_\fp(\sym_2^m,\lambda)^\ga 
:=\Cal{U}(\fg) \otimes_{\Cal{U}(\fp)}(S^m(\C^2)_\ga\boxtimes \C_\lambda),
\end{equation}
where $S^m(\C^2)_\ga$ stands for the $M$-representation
\begin{equation*}
S^m(\C^2)_\ga:=\C_{\ga} \otimes S^m(\C^2).
\end{equation*}

\vskip 0.1in

By the duality theorem (Theorem \ref{thm:duality}) and \eqref{eqn:dualsym}, we have 
\begin{equation}\label{eqn:dual2}
\Diff_G(I(\poly_2^m, \lambda)^\ga, I(\poly_2^\ell,  \nu)^\beta)
\simeq
\Hom_{\fg, P}(M_\fp(\sym_2^\ell,-\nu)^\gb, M_\fp(\sym_2^m, -\lambda)^\ga).
\end{equation}
Via the equivalence \eqref{eqn:dual2},
we shall classify and construct both intertwining differential operators 
\begin{equation*}
\D \in 
\Diff_G(I(\poly_2^m, \lambda)^\ga, I(\poly_2^\ell,  \nu)^\beta)
\end{equation*}
and $(\fg, P)$-homomorphisms
\begin{equation*}
\varphi \in \Hom_{\fg, P}(M_\fp(\sym_2^\ell,\nu)^\gb, M_\fp(\sym_2^m, \lambda)^\ga)
\end{equation*}
in Sections \ref{sec:classification} and \ref{sec:construction}.

%%%%%%%%%%%%%%%%%%%%%%%%%%%%%%%%%%%%%%%%%%
\section{Classification of parameters}
\label{sec:classification}

The aim of this section is to classify the 
parameters $(\ga, \gb; m, \ell; \lambda, \nu) \in 
(\Z/2\Z)^2\times (\Z_{\geq 0})^2 \times \C^2$ such that
\begin{equation*}
\Diff_G(I(\poly_2^m, \lambda)^\ga, I(\poly_2^\ell,  \nu)^\beta)\neq \{0\}.
\end{equation*}
By the duality theorem (Theorem \ref{thm:duality}), 
it is equivalent to classify the parameters such that
\begin{equation*}
 \Hom_{\fg, P}(M_\fp(\sym_2^\ell,-\nu)^\gb, M_\fp(\sym_2^m, -\lambda)^\ga) \neq \{0\}.
\end{equation*}
The classification is given in 
Theorems \ref{thm:IDO1} and \ref{thm:gP1}
for the former and  latter, respectively.

%%%%%%%%%%%%%%%%%%%%%%%%%%%%%%%%%%%%%%%%%%
\subsection{Classification of parameters}
\label{sec:main1}
We start by stating the main results. We put
\begin{align*}
\gL_1
&:=
\{
(\ga, \gb; m, \ell; \lambda,\nu) \in (\Z/2\Z)^2 \times (\Z_{\geq 0})^2 \times \C^2: 
\text{\eqref{eqn:Lambda1} holds.}\}\\[3pt]
\gL_2
&:=
\{
(\ga, \gb; m, \ell; \lambda,\nu) \in (\Z/2\Z)^2 \times (\Z_{\geq 0})^2 \times \C^2: 
\text{\eqref{eqn:Lambda2} holds.}\}
\end{align*}
\begin{alignat}{3}
&k:=\ell-m\in \Z_{\geq 0},
\quad
&&\beta-\alpha = k,
\quad
&&(\lambda,\nu) = (\tfrac{1}{2}(2-m-2k), \tfrac{1}{2}(2-m+k)) \label{eqn:Lambda1}\\[3pt]
&k:=m-\ell\in \Z_{\geq 0},
\quad
&&\beta-\alpha = k,
\quad
&&(\lambda,\nu)=(\tfrac{1}{2}(4+m-2k), \tfrac{1}{2}(4+m+k)).\label{eqn:Lambda2}
\end{alignat}
Here, we understand
$\beta-\ga = k$ for $k \in \Z$ as
\begin{equation}\label{eqn:parity}
\beta-\ga = (-1)^{k},
\end{equation}
where $\beta-\ga \in \{\pm\} \equiv \{\pm 1\}$
is defined by
\begin{equation*}
\beta-\ga=
\begin{cases}
+ & \text{if $\ga = \gb$},\\
- & \text{otherwise}.
\end{cases}
\end{equation*}

We set
\begin{equation*}
\gL:=\gL_1  \cup \gL_2.
\end{equation*}
The parameters are classified as follows.

\begin{thm}\label{thm:IDO1}
The following conditions on 
$(\ga, \gb; m,\ell; \lambda,\nu) \in (\Z/2\Z)^2 \times (\Z_{\geq 0})^2 \times \C^2$ 
are equivalent.
\begin{enumerate}
\item[\emph{(i)}]
$\Diff_G(I(\poly_2^m, \lambda)^\ga, I(\poly_2^\ell,  \nu)^\beta)\neq \{0\}$.

\item[\emph{(ii)}]
$\dim \Diff_G(I(\poly_2^m, \lambda)^\ga, I(\poly_2^\ell,  \nu)^\beta) =1$.

\item[\emph{(iii)}]
One of the following two cases holds:
\begin{enumerate}
\item[\emph{(iii-a)}] $(\ga, m, \lambda) = (\gb, \ell, \nu)$;
\item[\emph{(iii-b)}] $(\ga, \gb; m,\ell; \lambda,\nu) \in \gL$.
\end{enumerate}
\end{enumerate}
\end{thm}

The next theorem is the algebraic counterpart of Theorem \ref{thm:IDO1} via 
the duality theorem.% (Theorem \ref{thm:duality}).

\begin{thm}\label{thm:gP1}
The following conditions on 
$(\ga, \gb; m,\ell; \lambda,\nu) \in (\Z/2\Z)^2 \times (\Z_{\geq 0})^2 \times \C^2$ 
are equivalent.
\begin{enumerate}
\item[\emph{(i)}]
$\Hom_{\fg,P}(M_\fp(\sym_2^{\ell},-\nu)^{\gb},
M_\fp(\sym_2^m, -\lambda)^\ga) \neq \{0\}$.

\item[\emph{(ii)}]
$\dim \Hom_{\fg,P}(M_\fp(\sym_2^{\ell},-\nu)^{\gb},
M_\fp(\sym_2^m, -\lambda)^\ga) =1$.

\item[\emph{(iii)}]
One of the following two cases holds:
\begin{enumerate}
\item[\emph{(iii-a)}] $(\ga, m, \lambda) = (\gb, \ell, \nu)$;
\item[\emph{(iii-b)}] $(\ga, \gb; m,\ell; \lambda,\nu) \in \gL$.
\end{enumerate}
\end{enumerate}
\end{thm}

%%%%%%%%%%%%%%%%%%%%%%%%%%%%%%%%%%%%%%%%%%
\subsection{Strategy of the proof of Theorems \ref{thm:IDO1} and \ref{thm:gP1}}
\label{sec:proof1}
Our basic strategy to prove Theorems \ref{thm:IDO1} and \ref{thm:gP1} is to reduce 
the space $\Diff_G(I(\poly_2^m, \lambda)^\ga, I(\poly_2^\ell,  \nu)^\beta)$
as much as possible. More precisely, observe that 
as $M_0 \simeq SL(2,\R)$, we have  
\begin{equation}\label{eqn:std}
(M_0, \Ad_{\fn_-}, \fn_-)
\simeq
(SL(2,\R), \std, \C^2),
\end{equation}
whose equivalence is simply given by
\begin{equation*}
z_1 N_1^-+ z_2 N_2^- \mapsto z_1 e_1 + z_2 e_2,
\end{equation*}
where $e_j$ for $j=1,2$ are the standard unit vectors.
Equivalence \eqref{eqn:std} leads to
\begin{equation*}
(M_0, \Ad^m_{\fn_-}, S^m(\fn_-))
\simeq
(SL(2,\R), \sym_2^m, S^m(\C^2)).
\end{equation*}

\begin{obs}\label{obs:41}
By the duality theorem and \eqref{eqn:dualsym}, we have 
\begin{align}\label{eqn:equiv1}
\Diff_G(I(\poly_2^m, \lambda)^\ga, I(\poly_2^\ell,  \nu)^\beta)
&\simeq
\Hom_{\fg, P}(M_\fp(\sym_2^\ell,-\nu)^\gb, M_\fp(\sym_2^m, -\lambda)^\ga) \nonumber\\
&\simeq
\Hom_{P}(S^\ell(\C^2)_\beta\boxtimes \C_{-\nu}, M_\fp(\sym_2^m, -\lambda)^\alpha)
\nonumber\\
&=
\Hom_{MA}(S^\ell(\C^2)_\beta\boxtimes \C_{-\nu}, 
\left(M_\fp(\sym_2^m, -\lambda)^\alpha\right)^{\fn_+})\nonumber\\
&\subset
\Hom_{MA}(S^\ell(\C^2)_\beta\boxtimes \C_{-\nu}, M_\fp(\sym_2^m, -\lambda)^\alpha)
\nonumber\\
&=
\Hom_{MA}(S^\ell(\C^2)_\beta\boxtimes \C_{-\nu}, 
S(\fn_-)\otimes S^m(\C^2)_\alpha\otimes \C_{-\lambda}).
\end{align}

\noindent
Since the $MA$-representation $S(\fn_-)$ decomposes irreducibly into
\begin{equation*}
S(\fn_-)\vert_{MA} = \bigoplus_{k\in \Z_{\geq 0}}S^k(\fn_-),
\end{equation*}
the space of $MA$-homomorphisms in \eqref{eqn:equiv1}
can be further simplified to
\begin{align*}
\eqref{eqn:equiv1}
&=
\Hom_{MA}(S^\ell(\C^2)_\beta\boxtimes \C_{-\nu}, 
S(\fn_-)\otimes S^m(\C^2)_\alpha\otimes \C_{-\lambda})\\
&=
\bigoplus_{k \in \Z_{\geq 0}}
\Hom_{MA}(S^\ell(\C^2)_\beta\boxtimes \C_{-\nu}, 
S^k(\fn_-)\otimes S^m(\C^2)_\alpha\otimes \C_{-\lambda})\\
&\subset
\bigoplus_{k \in \Z_{\geq 0}}
\Hom_{M}(S^\ell(\C^2)_\beta, S^k(\fn_-)\otimes S^m(\C^2)_\alpha)\\
&\subset
\bigoplus_{k \in \Z_{\geq 0}}
\Hom_{M_0}(S^\ell(\C^2), S^k(\fn_-)\otimes S^m(\C^2))\\
&\simeq
\bigoplus_{k \in \Z_{\geq 0}}
\Hom_{SL(2,\R)}(S^\ell(\C^2), S^k(\C^2)\otimes S^m(\C^2)).
\end{align*}
\end{obs}

Based on Observation \ref{obs:41},
we first show the equivalence $\text{(i)}\Leftrightarrow \text{(iii)}$ 
of Theorems \ref{thm:IDO1} and \ref{thm:gP1}
in stages by proceeding with the following four steps.

\vskip 0.1in

\begin{enumerate}

\item[]
Step 1:
Classify  $(k, m, \ell) \in (\Z_{\geq 0})^3$ such that
\begin{equation*}
\Hom_{SL(2,\R)}(S^\ell(\C^2), S^k(\C^2)\otimes S^m(\C^2))\neq \{0\}.
\end{equation*}

\item[]
Step 2: 
Classify
$(\alpha,\beta;\ell,m,k) \in (\Z/2\Z)^2 \times (\Z_{\geq 0})^3$ such that
\begin{equation*}
\Hom_{M}(S^\ell(\C^2)_\beta, S^k(\fn_-)\otimes S^m(\C^2)_\alpha) \neq \{0\}.
\end{equation*}

\item[]
Step 3:
Classify
$(\alpha, \beta; m, k, \ell; \lambda,\nu)  \in 
(\Z/2\Z)^2 \times (\Z_{\geq 0})^3 \times \C^2$ such that
\begin{equation*}
\Hom_{MA}(S^\ell(\C^2)_\beta \otimes \C_{-\nu}, 
S^k(\fn_-)\otimes S^m(\C^2)_\ga\otimes \C_{-\lambda})\neq \{0\}.
\end{equation*}

\item[]
Step 4: 
Classify $(\ga, \gb; m, \ell; \lambda, \nu) 
\in (\Z/2\Z)^2 \times (\Z_{\geq 0})^2 \times \C^2$ such that
\begin{equation*}
\Hom_{\fg, P}(M_\fp(\sym_2^\ell,-\nu)^\beta, M_\fp(\sym_2^m, -\lambda)^\alpha) \neq \{0\}.
\end{equation*}

\end{enumerate}

After achieving the equivalence $\text{(i)}\Leftrightarrow \text{(iii)}$,
we prove the whole equivalence in Section \ref{sec:IDOpf}.

%%%%%%%%%%%%%%%%%%%%%%%%%%%%%%%%%%%%%%%%%%
\subsection{Step 1}
We start with the classification of $(k, m, \ell)$ such that
\begin{equation*}
\Hom_{SL(2,\R)}(S^\ell(\C^2), S^k(\C^2)\otimes S^m(\C^2))\neq \{0\}.
\end{equation*}

\begin{prop}\label{prop:Step1}
The following conditions on $(k, m, \ell) \in  (\Z_{\geq 0})^3$ are equivalent.
\begin{enumerate}
\item[\emph{(i)}] $\Hom_{SL(2,\R)}(S^\ell(\C^2), S^k(\C^2)\otimes S^m(\C^2))\neq \{0\}$.
\item[\emph{(ii)}] $m+k-\ell=2d$ for some $d \in \{0,\ldots, \min(k,m)\}$.
\end{enumerate}
\end{prop}

\begin{proof}
Recall the classical Clebsch–Gordan formula:
\begin{equation}\label{eqn:CG}
S^k(\C^2)\otimes S^m(\C^2)\vert_{SL(2,\R)}
\simeq 
\bigoplus_{0\leq d \leq \min(k,m)}S^{k+m-2d}(\C^2).
\end{equation}
The assertion immediately follows from \eqref{eqn:CG}.
\end{proof}

%%%%%%%%%%%%%%%%%%%%%%%%%%%%%%%%%%%%%%%%%%
\subsection{Step 2}
The next goal is to classify 
$(\alpha,\beta;\ell,m,k)$ such that
\begin{equation*}
\Hom_{M}(S^\ell(\C^2)_\beta, S^k(\fn_-)\otimes S^m(\C^2)_\alpha) \neq \{0\}.
\end{equation*}
A direct computation shows that
\begin{equation*}
(M,\Ad^k_{\fn_-},S^k(\fn_-))
\simeq (SL^{\pm}(2,\R), \sgn^{k}\otimes \sym_2^k, S^k(\C^2)).
\end{equation*}
Thus, as an $M$-representation, we have
\begin{equation*}
S^k(\fn_-)\otimes S^m(\C^2)_\ga
\simeq
S^k(\C^2) \otimes S^m(\C^2)_{\ga+k},
\end{equation*}
where $\ga + k \in \{\pm\}$ is defined as 
\begin{equation}\label{eqn:ak}
\ga+k=
\begin{cases}
+ & \text{if $\ga=(-1)^k$},\\
- & \text{if $\ga=(-1)^{k+1}$}.
\end{cases}
\end{equation}

\begin{prop}\label{prop:Step2}
The following conditions on 
$(\alpha,\beta;\ell,m,k)\in (\Z/2\Z)^2 \times (\Z_{\geq 0})^3$
are equivalent.
\begin{enumerate}
\item[\emph{(i)}] 
$\Hom_{M}(S^\ell(\C^2)_\beta, S^k(\fn_-)\otimes S^m(\C^2)_\alpha) \neq \{0\}$.
\item[\emph{(ii)}] Both of the following conditions hold:
\begin{enumerate}
\item[\emph{(ii-a)}] $m+k-\ell=2d$ for some $d \in \{0,\ldots, \min(k,m)\}$;
\item[\emph{(ii-b)}] $\beta -\alpha=  k$.
\end{enumerate}
\end{enumerate}
\end{prop}

\begin{proof}
The proposition readily follows from Proposition \ref{prop:Step1} and the above 
observation.
\end{proof}

%%%%%%%%%%%%%%%%%%%%%%%%%%%%%%%%%%%%%%%%%%
\subsection{Step 3}
As a third step,
we wish to classify 
$(\alpha, \beta; m, k, \ell; \lambda,\nu)$
such that
\begin{equation*}
\Hom_{MA}(S^\ell(\C^2)_\beta \otimes \C_{-\nu}, 
S^k(\fn_-)\otimes S^m(\C^2)_\ga\otimes \C_{-\lambda})\neq \{0\}.
\end{equation*}

\noindent
It follows from \eqref{eqn:nR} that $\wH_0$ acts on $S^k(\fn_-)$ 
by weight $-\frac{3}{2}k$, which shows that, as an $MA$-module, we have 
\begin{equation}\label{eqn:Step2}
S^k(\fn_-)\otimes S^m(\C^2)_\alpha \otimes \C_{-\lambda}
\simeq
\left( S^k(\C^2) \otimes S^m(\C^2)_{\alpha +k} \right) 
\boxtimes \C_{-(\lambda+\frac{3}{2}k)}.
\end{equation}

\begin{prop}\label{prop:Step3}
The following conditions on 
$(\alpha, \beta; m, k, \ell; \lambda,\nu)\in (\Z/2\Z)^2 \times (\Z_{\geq 0})^3 \times \C^2$ 
are equivalent:
\begin{enumerate}
\item[\emph{(i)}] 
$\Hom_{MA}(S^\ell(\C^2)_\beta \otimes \C_{-\nu}, 
S^k(\fn_-)\otimes S^m(\C^2)_\ga\otimes \C_{-\lambda})\neq \{0\}$;
\item[\emph{(ii)}] All of the following conditions hold:
\begin{enumerate}
\item[\emph{(ii-a)}] $m+k-\ell=2d$ for some $d \in \{0,\ldots, \min(k,m)\}$;
\item[\emph{(ii-b)}] $\beta-\alpha=k$;
\item[\emph{(ii-c)}] $\nu-\lambda = \frac{3}{2}k$.
\end{enumerate}
\end{enumerate}
\end{prop}

\begin{proof}
Proposition \ref{prop:Step2} together with \eqref{eqn:Step2} concludes the proposition.
\end{proof}

%%%%%%%%%%%%%%%%%%%%%%%%%%%%%%%%%%%%%%%%%%
\subsection{Step 4}
\label{sec:Step4}
In the final step,
we wish to classify $(\ga, \gb; m,\ell; \lambda, \nu)$ such that
\begin{equation}\label{eqn:Step4}
\Hom_{\fg, P}(M_\fp(\sym_2^\ell,-\nu)^\beta, M_\fp(\sym_2^m, -\lambda)^\alpha) \neq \{0\}.
\end{equation}

It follows from Observation \ref{obs:41} that if \eqref{eqn:Step4} holds, then
the parameters $(\ga, \gb; m,\ell; \lambda, \nu)$ satisfy
the condition (ii) of Proposition \ref{prop:Step3} for some $k \in \Z_{\geq 0}$. 
Further, by \cite[Cor.\ 2.16]{KuOr19}, 
the parameter $\ga$ can be chosen freely as far as $\beta-\ga \in \Z_{\geq 0}$.
Thus, it suffices to classify 
$(m,k,\ell; \lambda) \in  (\Z_{\geq 0})^3 \times \C$ 
such that
\begin{equation*}
\Hom_{\fg, P}(M_\fp(\sym_2^\ell,-(\lambda+\tfrac{3}{2}k))^{\alpha+k},
M_\fp(\sym_2^m, -\lambda)^\alpha) \neq \{0\}
\end{equation*}
with condition
\begin{equation*}
\frac{m+k-\ell}{2} \in \Z\cap [0,\min(k,m)].
\end{equation*}

\begin{thm}\label{thm:Verma}
The following conditions on 
$(m,k,\ell; \lambda) \in  (\Z_{\geq 0})^3 \times \C$ are equivalent.
\begin{enumerate}
\item[\emph{(i)}]
$\Hom_{\fg,P}(M_\fp(\sym_2^{\ell},-(\lambda+\tfrac{3}{2}k))^{\ga+k},
M_\fp(\sym_2^m, -\lambda)^\ga) \neq \{0\}$.

\item[\emph{(ii)}]
One of the following cases holds:
\begin{enumerate}
\item[\emph{(ii-a)}] $(k, \ell) = (0, m)$ and $\lambda \in \C$;
\item[\emph{(ii-b)}]
$\ell =m+ k$ and $\lambda = \tfrac{1}{2}(2-m-2k)$;
\item[\emph{(ii-c)}] 
$\ell=m-k$ and $\lambda = \tfrac{1}{2}(4+m-2k)$.
\end{enumerate}
\end{enumerate}
\end{thm}

\begin{rem}\label{rem:CPRV}
Theorem \ref{thm:Verma} shows that only the largest component
$S^{m+k}(\C^2)$ and the smallest component  $S^{m-k}(\C^2)$
of the tensor product decomposition of $S^k(\C^2)\otimes S^m(\C^2)$
in \eqref{eqn:CG} contribute to the construction of intertwining differential
operators. These components $S^{m+k}(\C^2)$ and $S^{m-k}(\C^2)$ 
are  called 
the \emph{Cartan component} and \emph{PRV component} of 
$S^k(\C^2)\otimes S^m(\C^2)$, respectively (cf.\ \cite{Kumar10}),
where PRV comes from the names 
K.\ R.\ Parthasarathy, R.\ Ranga Rao, and V.\ S.\ Varadarajan.

 If  $m=0$, then 
we regard $S^k(\C^2)=S^k(\C^2)\otimes S^0(\C^2)$  as both the Cartan component
and PRV component.
\end{rem}

We prove Theorem \ref{thm:Verma} by
utilizing the idea of standard map between generalized Verma modules. 
Then we first rewrite the generalized Verma modules as in Section \ref{sec:Std1}.

Take a Cartan subalgebra $\fh$ of $\fg$ such that $\fh:=\{\diag(a_1, a_2, a_3): 
\sum_{j=1}^3a_j=0\}$. Let $\gD\equiv \gD(\fg,\fh)$ 
denote the set of roots of $\fg$ with respect to $\fh$.
As usual, we take sets of positive roots $\gD^+$ and simple roots $\Pi$ as
\begin{equation*}
\gD^+=\{ \eps_1 - \eps_2, \eps_2 - \eps_3, \eps_1 - \eps_3\}
\end{equation*}
and
\begin{equation*}
\Pi = \{ \eps_1 - \eps_2, \eps_2 - \eps_3\}.
\end{equation*}
We realize $\fh^*$ as a subspace of $\C^{3}$ by identifying 
$\eps_j$ with the canonical basis $e_j$,
and write elements in $\fh^*$ in coordinates.
For instance, half the sum of the positive roots 
$\rho = (1/2)\sum_{\ga \in \gD^+}\ga$ is given by
\begin{equation*}
\rho = \frac{1}{2}(2, 0, -2)=(1,0,-1).
\end{equation*}
Likewise, the differential $d\chi \in \fh^*$ 
of the character $\chi$ of $A$  with $\lambda = 1$ defined in 
\eqref{eqn:chi} is expressed as
\begin{equation}\label{eqn:dchi}
d\chi= \frac{1}{3}(2, -1, -1),
\end{equation}
so that $d\chi^\lambda$ is given by
\begin{equation*}
d\chi^\lambda=\lambda \cdot d\chi =\frac{\lambda}{3}(2, -1, -1).
\end{equation*}

A root reflection with respect to $\ga \in \gD$ is denoted by $s_\ga$. 
The Weyl group $W$ of $\fg$ is then given by
\begin{equation}\label{eqn:Weyl}
W = \{e, 
s_{\eps_1-\eps_2}, \,
s_{\eps_2-\eps_3}, \,
s_{\eps_2-\eps_3}s_{\eps_1-\eps_2}, \,
s_{\eps_1-\eps_2}s_{\eps_2-\eps_3}, \,
s_{\eps_1-\eps_2}s_{\eps_2-\eps_3}s_{\eps_1-\eps_2}\}
\end{equation}
with
\begin{equation}\label{eqn:13}
s_{\eps_1-\eps_2}s_{\eps_2-\eps_3}s_{\eps_1-\eps_2}
= s_{s_{\eps_1-\eps_2}(\eps_2-\eps_3)}
=s_{\eps_1-\eps_3}
=s_{s_{\eps_2-\eps_3}(\eps_1-\eps_2)}
=s_{\eps_2-\eps_3}s_{\eps_1-\eps_2}s_{\eps_2-\eps_3}.
\end{equation}

\vskip 0.1in

Recall from Section \ref{sec:Std1} that, for $\mu \in \mathbf{P}^+_{\fl}$
for $\fl = \fm\oplus\fa$,
we write
\begin{equation*}
N_\fp(\mu)=\Cal{U}(\fg)\otimes_{\Cal{U}(\fp)} E(\mu-\rho),
\end{equation*}
where $E(\mu-\rho)$ is the finite-dimensional simple $\Cal{U}(\fl)$-module
with highest weight $\mu-\rho$. 

Write
\begin{equation*}
\omega_1:=\frac{1}{2}(0, 1, -1).
\end{equation*} 
Then the generalized Verma modules in consideration can be expressed as
\begin{align*}
M_\fp(\sym_2^m,-\lambda)^\ga
&=N_\fp(m\omega_1 -\lambda d\chi+\rho)^\ga, \\[5pt]
M_\fp(\sym_2^\ell, -(\lambda +\tfrac{3}{2}k))^\gb
&=N_\fp(\ell\omega_1 -(\lambda +\tfrac{3}{2}k)d\chi+\rho)^\gb.
\end{align*}
Now we wish to classify $(m,k,\ell; \lambda)$ such that
\begin{equation}\label{eqn:HomVerma}
\Hom_{\fg,P}
%\left(
(N_\fp(\ell\omega_1 -(\lambda +\tfrac{3}{2}k)d\chi+\rho)^{\ga+k},
N_\fp(m\omega_1 -\lambda d\chi+\rho)^\ga)
%\right) 
\neq \{0\}
\end{equation}
with condition
\begin{equation}\label{eqn:cond}
\frac{m+k-\ell}{2} \in \Z\cap [0,\min(k,m)].
\end{equation}

We start with a necessary condition of \eqref{eqn:HomVerma}.

\begin{prop}\label{prop:Verma}
If \eqref{eqn:HomVerma} holds, then one of the following conditions holds.
\begin{enumerate}
\item[\emph{(a)}] $(k, \ell) = (0, m)$ and $\lambda \in \C$.
\item[\emph{(b)}] 
$\ell = m+k$ and $\lambda = \tfrac{1}{2}(2-m-2k)$.% (Cartan component)
\item[\emph{(c)}] 
$\ell=m-k$ and $\lambda = \tfrac{1}{2}(4+m-2k)$.% (PRV component)
\end{enumerate}
\end{prop}

\begin{proof}
By assumption,
the infinitesimal character of 
$N_\fp(\ell\omega_1 -(\lambda +\tfrac{3}{2}k)d\chi+\rho)^{\ga+k}$
agrees with that of $N_\fp(m\omega_1 -\lambda d\chi+\rho)^\ga$. 
Thus, there exists $w \in W$ such that
\begin{equation}\label{eqn:inf}
\ell\omega_1 -(\lambda +\tfrac{3}{2}k)d\chi+\rho
=w(m\omega_1 -\lambda d\chi+\rho).
\end{equation}

For a simple expression of both sides of \eqref{eqn:inf}, we normalize $\lambda$ by
\begin{equation*}
\wlambda:=\tfrac{2}{3}\lambda.
\end{equation*}
Then the weights $m\omega_1 -\lambda d\chi+\rho$ and 
$\ell\omega_1 -(\lambda +\tfrac{3}{2}k)d\chi+\rho$ are given in coordinates as
\begin{align*}
m\omega_1 -\lambda d\chi+\rho
&=m\omega_1 -\tfrac{3}{2}\wlambda d\chi+\rho\\
&=\tfrac{m}{2}(0, 1, -1) - \tfrac{\wlambda}{2}(2,-1,-1) + (1,0,-1)\\
&=
(1-\wlambda, \tfrac{1}{2}(\wlambda+m),\tfrac{1}{2}(\wlambda-m-2))
\end{align*}
and
\begin{align*}
\ell\omega_1 -(\lambda +\tfrac{3}{2}k)d\chi+\rho
&=\ell\omega_1 -\tfrac{3}{2}(\wlambda +k)d\chi+\rho\\
&=\tfrac{\ell}{2}(0,1,-1)-\tfrac{\wlambda+k}{2}(2,-1,-1)+(1,0,-1)\\
&=
(1-\wlambda-k, \tfrac{1}{2}(\wlambda+k+\ell), \tfrac{1}{2}(\wlambda+k-\ell-2)).
\end{align*}
Thus \eqref{eqn:inf}  amounts to
\begin{equation*}
(1-\wlambda-k, \tfrac{1}{2}(\wlambda+k+\ell), \tfrac{1}{2}(\wlambda+k-\ell-2))
=w(1-\wlambda, \tfrac{1}{2}(\wlambda+m),\tfrac{1}{2}(\wlambda-m-2)).
\end{equation*}

\noindent
Since the Weyl group element $s_{\eps_i-\eps_j}$ acts on $\fh^*\subset \C^3$
as the exchange of the $i$th entry with the $j$th entry, 
it follows from \eqref{eqn:Weyl} that one of the following cases holds.

\begin{enumerate}[(1)]

\item $1-\wlambda =1-\wlambda-k$.

\vskip 0.05in

\item $1-\wlambda = \tfrac{1}{2}(\wlambda+k+\ell)$.

\vskip 0.05in

\item $1-\wlambda = \tfrac{1}{2}(\wlambda+k-\ell-2)$.

\vskip 0.05in

\end{enumerate}
\noindent
We consider these cases separately.

\vskip 0.1in

%%%%%%%%%%%%%%%%%%%%%%%%%%%%%%%%%%%%%%%%%%
\noindent
Case (1): 
Since $1-\wlambda =1-\wlambda-k$,
we have $k=0$.
It then follows from \eqref{eqn:cond} that  $m=\ell$. Thus,
\begin{equation*}
N_\fp(\ell\omega_1 -(\lambda +\tfrac{3}{2}k)d\chi+\rho)^{\ga+k}=
N_\fp(m\omega_1 -\lambda d\chi+\rho)^\ga.
\end{equation*}
In this case, for any $\lambda\in \C$, we have
\begin{equation}\label{eqn:Case1}
%\Hom_{\fg}
%\left(
%N_\fp(\ell\omega_1 -(\lambda +\tfrac{3}{2}k)d\chi+\rho),
%N_\fp(m\omega_1 -\lambda d\chi+\rho)
%\right)
%&=
0\neq \id \in 
\Hom_{\fg, P}
%\left
(
N_\fp(m\omega_1 -\lambda d\chi+\rho)^{\ga},
N_\fp(m\omega_1 -\lambda d\chi+\rho)^\ga
%\right
),
\end{equation}
where $\id$ denotes the identity map.
\vskip 0.1in

%%%%%%%%%%%%%%%%%%%%%%%%%%%%%%%%%%%%%%%%%%
\noindent
Case (2): 
In this case, we have either
\begin{equation}\label{eqn:2a}
\begin{aligned}
1-\wlambda &= \tfrac{1}{2}(\wlambda+k+\ell),\\
 \tfrac{1}{2}(\wlambda+m) &=1-\wlambda-k,\\
\tfrac{1}{2}(\wlambda-m-2)&=\tfrac{1}{2}(\wlambda+k-\ell-2),
\end{aligned}
\end{equation}
or
\begin{equation}\label{eqn:2b}
\begin{aligned}
1-\wlambda &= \tfrac{1}{2}(\wlambda+k+\ell),\\
 \tfrac{1}{2}(\wlambda+m) &=\tfrac{1}{2}(\wlambda+k-\ell-2),\\
\tfrac{1}{2}(\wlambda-m-2)&=1-\wlambda-k.
\end{aligned}
\end{equation}

If \eqref{eqn:2a} holds, then a direct computation shows that
\begin{equation}\label{eqn:2a2}
\begin{aligned}
\wlambda&=\frac{1}{3}(2-m-2k),\\
\ell&=m+k.
\end{aligned}
\end{equation}

If \eqref{eqn:2b} holds, then $\ell=k-m-2$. 
Since $\ell \geq 0$, this implies that $\min(k,m)=m$.
It then follows from \eqref{eqn:cond} that there exists $d \in \{0,\ldots, m\}$ such that
$\ell = k+m-2d$. Thus,
\begin{equation*}
k-m-2=\ell =k+m-2d \geq k-m,
\end{equation*}
which is a contradiction. Therefore, \eqref{eqn:2b} does not occur. 

\vskip 0.1in

\item 
Case (3): 
As in Case (2), in this case, we have either
\begin{equation}\label{eqn:3a}
\begin{aligned}
1-\wlambda &= \tfrac{1}{2}(\wlambda+k-\ell-2),\\
\tfrac{1}{2}(\wlambda+m) &=\tfrac{1}{2}(\wlambda+k+\ell),\\
\tfrac{1}{2}(\wlambda-m-2)&=1-\wlambda-k,
\end{aligned}
\end{equation}
or
\begin{equation}\label{eqn:3b}
\begin{aligned}
1-\wlambda &= \tfrac{1}{2}(\wlambda+k-\ell-2),\\
\tfrac{1}{2}(\wlambda+m)&=1-\wlambda-k,\\
\tfrac{1}{2}(\wlambda-m-2) &=\tfrac{1}{2}(\wlambda+k+\ell).
\end{aligned}
\end{equation}

If \eqref{eqn:3a} holds, then we have
\begin{equation}\label{eqn:3a2}
\begin{aligned}
\wlambda&=\frac{1}{3}(4+m-2k),\\
\ell&=m-k.
\end{aligned}
\end{equation}

If \eqref{eqn:3b} holds, then $\ell=-(m+k+2) < 0$, which contradicts 
the condition that  $\ell \geq 0$. 
Thus, this case does not happen.

The conditions \eqref{eqn:Case1}, \eqref{eqn:2a2},
and \eqref{eqn:3a2} conclude the proposition.
\end{proof}

Now we are ready to prove Theorem \ref{thm:Verma}.

\begin{proof}[Proof of Theorem \ref{thm:Verma}]
By Proposition \ref{prop:Verma}, it suffices to show that 
Condition (ii) of Theorem \ref{thm:Verma} implies \eqref{eqn:Step4},
equivalently, \eqref{eqn:HomVerma}.
Since we already checked that the identity map $\id$ exists in Condition (ii-a) in the proof of Proposition \ref{prop:Verma}, we only consider Conditions (ii-b) and (ii-c). 
We are going to show that the standard map is non-zero for these cases.
 
 \vskip 0.1in
 
\noindent
Condition (ii-b): 
As $\ell = m+k$ and $\lambda = \tfrac{1}{2}(2-m-2k)$, we have
\begin{align*}
\ell\omega_1 -(\lambda +\tfrac{3}{2}k)d\chi+\rho 
&=\tfrac{1}{3}(m-k+1, m+2k+1, -(2m+k+2)),\\
m\omega_1 -\lambda d\chi+\rho
&=\tfrac{1}{3}(m+2k+1, m-k+1, -(2m+k+2)).
\end{align*}
Thus,
\begin{equation*}
\ell\omega_1 -(\lambda +\tfrac{3}{2}k)d\chi+\rho 
=s_{\eps_1-\eps_2}(m\omega_1 -\lambda d\chi+\rho).
\end{equation*}
Since
\begin{equation*}
\IP{m\omega_1 -\lambda d\chi+\rho}{\eps_1-\eps_2}
=k \in \Z_{\geq 0},
\end{equation*}
the positive root $\eps_1-\eps_2$ links 
$m\omega_1 -\lambda d\chi+\rho$ to 
$\ell\omega_1 -(\lambda +\tfrac{3}{2}k)d\chi+\rho$.
Thus, it follows from Theorem \ref{thm:link} that 
there exists a non-zero $\fg$-homomorphism 
\begin{equation*}
\varphi \colon
N(\ell\omega_1 -(\lambda +\tfrac{3}{2}k)d\chi+\rho)
\To
N(m\omega_1 -\lambda d\chi+\rho)
\end{equation*}
between (ordinary) Verma modules
$N(\ell\omega_1 -(\lambda +\tfrac{3}{2}k)d\chi+\rho)$
and
$N(m\omega_1 -\lambda d\chi+\rho)$.
Further, clearly, $\eps_1-\eps_2$ is the only linking positive root.
Thus, by Theorem \ref{thm:Boe}, 
the standard map
\begin{equation*}
\varphi_{\std}\colon
N_\fp(\ell\omega_1 -(\lambda +\tfrac{3}{2}k)d\chi+\rho)^{\ga+k}
\to
N_\fp(m\omega_1 -\lambda d\chi+\rho)^\ga
\end{equation*}
associated with $\varphi$ is non-zero. 
 
 \vskip 0.1in

\noindent 
Condition (ii-c):
As
$\ell=m-k$ and $\lambda = \tfrac{1}{2}(4+m-2k)$,
we have
\begin{align*}
\ell\omega_1 -(\lambda +\tfrac{3}{2}k)d\chi+\rho 
&=
\tfrac{1}{3}(-(m+k+1), 2m-k+2, -m+2k-1)),\\
m\omega_1 -\lambda d\chi+\rho
&=\tfrac{1}{3}(-m+2k-1, 2m-k+2, -(m+k+1)).
\end{align*}
Thus,
\begin{equation*}
\ell\omega_1 -(\lambda +\tfrac{3}{2}k)d\chi+\rho 
=s_{\eps_1-\eps_3}(m\omega_1 -\lambda d\chi+\rho).
\end{equation*}
Since
\begin{equation*}
\IP{m\omega_1 -\lambda d\chi+\rho}{\eps_1-\eps_3}
=k \in \Z_{\geq 0},
\end{equation*}
the positive root $\eps_1-\eps_3$ links 
$m\omega_1 -\lambda d\chi+\rho$ to 
$\ell\omega_1 -(\lambda +\tfrac{3}{2}k)d\chi+\rho$.
As in \eqref{eqn:13},
we have 
$s_{\eps_1-\eps_2}s_{\eps_2-\eps_3}s_{\eps_1-\eps_2}
=s_{\eps_2-\eps_3}s_{\eps_1-\eps_2}s_{\eps_2-\eps_3}
=s_{\eps_1-\eps_3}$. However, one can check that under the condition $m-k \geq 0$,
neither $(\eps_1-\eps_2, \eps_2-\eps_3, \eps_1-\eps_2)$ nor 
$(\eps_2-\eps_3, \eps_1-\eps_2, \eps_2-\eps_3)$ links 
$m\omega_1 -\lambda d\chi+\rho$ to 
$\ell\omega_1 -(\lambda +\tfrac{3}{2}k)d\chi+\rho$.
Thus, $\eps_1-\eps_3$ is the only linking positive root.
Then, by Theorems \ref{thm:link} and \ref{thm:Boe}, 
there exists a non-zero $\fg$-homomorphism
\begin{equation*}
\varphi\colon
N(\ell\omega_1 -(\lambda +\tfrac{3}{2}k)d\chi+\rho)
\to
N(m\omega_1 -\lambda d\chi+\rho)
\end{equation*}
between the Verma modules and 
the associated standard map
\begin{equation*}
\varphi_{\std}\colon
N_\fp(\ell\omega_1 -(\lambda +\tfrac{3}{2}k)d\chi+\rho)^{\ga+k}
\to
N_\fp(m\omega_1 -\lambda d\chi+\rho)^\ga
\end{equation*}
is non-zero. Now the proposed assertion follows.
\end{proof}

%\vskip 1in

%%%%%%%%%%%%%%%%%%%%%%%%%%%%%%%%%%%%%%%%%%
\subsection{Proof of Theorems \ref{thm:IDO1} and \ref{thm:gP1}}\label{sec:IDOpf}
To finish this section, we are going to prove 
Theorems \ref{thm:IDO1} and \ref{thm:gP1}.

\begin{proof}[Proof of Theorems \ref{thm:IDO1} and \ref{thm:gP1}]
By the duality theorem, it suffices to show Theorem \ref{thm:gP1}.
The equivalence of (i) and (iii) follows from Theorem \ref{thm:Verma} 
and its preceding arguments.
To show the equivalence of (ii) and (iii), 
for $\gb-\ga=\ell-m$ and $\nu-\lambda =\frac{3}{2}|\ell-m|$,
we have 
\begin{align*}
&\dim \Hom_{\fg, P}(M_\fp(\sym_2^\ell,-\nu)^\gb, M_\fp(\sym_2^m, -\lambda)^\ga) \\
&\hspace{5cm}
\leq 
\dim \Hom_{SL(2,\R)}(S^\ell(\C^2), S^{|\ell-m|}(\C^2)\otimes S^m(\C^2)) \leq 1.
\end{align*}
Now the desired equivalence follows.
\end{proof}

%%%%%%%%%%%%%%%%%%%%%%%%%%%%%%%%%%%%%%%%%%

\begin{rem}
If \eqref{eqn:HomVerma} holds, then the generalized Verma module
$N_\fp(m\omega_1 -\lambda d\chi+\rho)$ is reducible.
Indeed, it follows from a criterion of Bai--Xiao in \cite[Thm.\ 5.1]{BX21}
that generalized Verma modules $N_\fp(m\omega_1 -\lambda d\chi+\rho)$
for $(m,\ell; \lambda,\nu)$ satisfying \eqref{eqn:Lambda1} or 
\eqref{eqn:Lambda2} are reducible.
\end{rem}

%%%%%%%%%%%%%%%%%%%%%%%%%%%%%%%%%%%%%%%%%%
\section{Explicit formulas of $\D$ and $\varphi$}
\label{sec:construction}

%%%%%%%%%%%%%%%%%%%%%%%%%%%%%%%%%%%%%%%%%%%%%%%

Our next aim is to give explicit formulas for intertwining differential operators
$\D$ and $(\fg, P)$-homomorphisms $\varphi$.
Those formulas are given in Theorems \ref{thm:IDO2} and \ref{thm:gP2}, respectively.

%\vskip 0.1in

In the following,  we define
multiplication on the tensor product $A \otimes B \otimes C$
of three rings $A$, $B$, and $C$, componentwise, that is,
for $a_j \in A$, $b_j \in B$, and $c_j \in C$ for $j=1,2$, we define
$(a_1\otimes b_1 \otimes c_1)(a_2 \otimes b_2\otimes c_2)$ by
\begin{equation*}
(a_1\otimes b_1 \otimes c_1)(a_2 \otimes b_2\otimes c_2) := 
a_1a_2 \otimes b_1 b_2 \otimes c_1c_2.
\end{equation*}
If the rings in concern have unity $1$, we skip writing $1$ in the component.
For instance, we abbreviate $a \otimes b \otimes 1$ to $a \otimes b$. 

As usual, for $z, w \in \C$, 
the binomial coefficient $\binom{z}{w}$ is defined as 
\begin{equation*}
\binom{z}{w}:=
\begin{cases}
\frac{z(z-1) \cdots (z-w+1)}{w!} & \text{if $w \in 1+\Z_{\geq 0}$},\\
1 & \text{if $w=0$},\\
0 & \text{otherwise}.
\end{cases}
\end{equation*}

%%%%%%%%%%%%%%%%%%%%%%%%%%%%%%%%%%%%%%%%%%%%%%%
\subsection{Explicit formulas of $\D$}
\label{sec:main2a}
We start with the explicit formulas of intertwining differential operators 
$\D$. As in \eqref{eqn:DN}, 
we understand 
$\D \in 
\Diff_G(I(\poly_2^m, \lambda)^\ga, I(\poly_2^\ell,  \nu)^\beta)$
as a linear map
\begin{equation}\label{eqn:D51}
\D\colon C^\infty(\R^2)\otimes \Pol^m(\C^2) \To 
C^{\infty}(\R^2)\otimes \Pol^\ell(\C^2)
\end{equation}
via the diffeomorphism
\begin{equation}\label{eqn:coord}
\R^2 
\stackrel{\sim}{\To} N_-, \quad (x_1, x_2) 
\mapsto \exp(x_1 N^-_1 +x_2N^-_2).\\[3pt]
\end{equation}
In order to distinguish between $\Pol^m(\C^2)$ on the source
and $\Pol^\ell(\C^2)$ on the target in \eqref{eqn:D51}, we write
\begin{equation*}
\Pol^\ell(\C^2):=\C^m[u_1,u_2]
\quad
\text{and}
\quad
\Pol^\ell(\C^2):=\C^\ell[v_1,v_2].
\end{equation*}
Then we consider
\begin{equation}\label{eqn:D51a}
\D\colon C^\infty(\R^2)\otimes \C^m[u_1, u_2] \To 
C^{\infty}(\R^2)\otimes \C^\ell[v_1, v_2].
\end{equation}

With respect to the coordinate \eqref{eqn:coord},
the infinitesimal right translation $dR(N_j^-)$ for $j=1,2$ is given by
$dR(N_j^-) = \tfrac{\partial}{\partial x_j}$.
It then follows from the duality theorem
(Theorem \ref{thm:duality}) that 
 intertwining differential operators
 $\D$ in \eqref{eqn:D51a}
 have constant coefficients only. Thus \eqref{eqn:DN3} reads
\begin{equation}\label{eqn:D51b}
\Diff_G(I(\poly_2^m, \lambda)^\ga, I(\poly_2^\ell,  \nu)^\beta)
\subset
\C\left[\tfrac{\partial}{\partial x_1}, \tfrac{\partial}{\partial x_2}\right]
 \otimes \C^m[u_1, u_2]^\vee
 \otimes \C^\ell[v_1,v_2].
\end{equation}
Via the identification
$\C^m[u_1, u_2]^\vee 
\simeq \C^m[\tfrac{\partial}{\partial u_1}, \tfrac{\partial}{\partial u_2}]$,
we further regard \eqref{eqn:D51b} as
\begin{equation}
\Diff_G(I(\poly_2^m, \lambda)^\ga, I(\poly_2^\ell,  \nu)^\beta)
\subset
\C\left[\tfrac{\partial}{\partial x_1}, \tfrac{\partial}{\partial x_2}\right]
 \otimes \C^m\left[\tfrac{\partial}{\partial u_1}, \tfrac{\partial}{\partial u_2}\right]
 \otimes \C^\ell[v_1,v_2].\label{eqn:DC}
\end{equation}

For $ \frac{\partial^m}{\partial u_1^{m-j} \partial u_2^j} \in 
\C^m[\tfrac{\partial}{\partial u_1}, \tfrac{\partial}{\partial u_2}]$, we write
\begin{equation*}
\widetilde{\frac{\partial^m}{\partial u_1^{m-j} \partial u_2^j}}
:=\frac{1}{(m-j)!j!}\cdot \frac{\partial^m}{\partial u_1^{m-j} \partial u_2^j}.
\end{equation*}

\vskip 0.1in

%%%%%%%%%%%%%%%%%%%%%%%%%%%%%%%%%%%%%%%%%%
For $(\ga, \gb; m,\ell; \lambda,\nu) \in \gL_1$ with $k:=\ell-m$, 
we define 
\begin{equation*}
\Cal{D}^{m+k}_{m,k} \equiv \Cal{D}_{m,\ell-m}^{\ell}  \colon C^\infty(\R^2)\otimes \C^m[u_1, u_2]
\To
C^\infty(\R^2)\otimes \C^{m+k}[v_1, v_2]
\end{equation*}
by
\begin{equation}\label{eqn:IDO1}
\Cal{D}^{m+k}_{m,k} :=
\frac{1}{m!}
\left(\tfrac{\partial}{\partial u_1}\otimes v_1 + \tfrac{\partial}{\partial u_2}\otimes v_2\right)^m
\left(\tfrac{\partial}{\partial x_1}\otimes v_1 + \tfrac{\partial}{\partial x_2}\otimes v_2\right)^k,
\end{equation}
which amounts to
\begin{equation*}
\Cal{D}^{m+k}_{m,k}
=
\sum_{j=0}^m\sum_{r=0}^k
\binom{k}{r}\frac{\partial^k}{\partial x_1^{k-r} \partial x_2^r}
\otimes \widetilde{\frac{\partial^m}{\partial u_1^{m-j}\partial u_2^j} }
\otimes v_1^{m+k-j-r}v_2^{j+r}.
\end{equation*}
For 
\begin{equation*}
f(x_1, x_2) 
= \sum_{j=0}^m f_j(x_1,x_2) 
\otimes u_1^{m-j}u_2^j \in C^\infty(\R^2)\otimes \C^m[u_1, u_2], 
\end{equation*}
we have 
\begin{equation}\label{eqn:20250811}
\Cal{D}^{m+k}_{m,k}f(x_1, x_2) = 
\sum_{j=0}^m\sum_{r=0}^k
\binom{k}{r}\frac{\partial^kf_j}{\partial x_1^{k-r} \partial x_2^r}(x_1, x_2)
\otimes v_1^{m+k-j-r}v_2^{j+r}.\\[3pt]
\end{equation}

%%%%%%%%%%%%%%%%%%%%%%%%%%%%%%%%%%%%%%%%%%
Similarly, for $(\ga, \gb; m, \ell; \lambda,\nu) \in \gL_2$ with $k:=m-\ell$,
we define
\begin{equation*}
\Cal{D}^{m-k}_{m,k} \equiv \Cal{D}_{m,m-\ell}^{\ell} \colon C^\infty(\R^2)\otimes \C^m[u_1, u_2]
\To
C^\infty(\R^2)\otimes \C^{m-k}[v_1, v_2]
\end{equation*}
by
\begin{equation}\label{eqn:IDO2}
\Cal{D}^{m-k}_{m,k} :=
\frac{1}{m!}
\left(\tfrac{\partial}{\partial u_1}\otimes v_1 + \tfrac{\partial}{\partial u_2}\otimes v_2\right)^{m-k}
\left(\tfrac{\partial}{\partial x_1}\otimes \tfrac{\partial}{\partial u_2} - \tfrac{\partial}{\partial x_2}\otimes \tfrac{\partial}{\partial u_1}\right)^k,
\end{equation}
which amounts to
\begin{align*}
\Cal{D}^{m-k}_{m,k}
&=
\sum_{j=0}^{m-k}\sum_{r=0}^k
(-1)^r
\frac{\binom{m-k}{j}\binom{k}{k-r}}{\binom{m}{k+j-r}}
\frac{\partial^k}{\partial x_1^{k-r} \partial x_2^r}
\otimes
\widetilde{\frac{\partial^m}{\partial u_1^{m-k-j+r}\partial u_2^{k+j-r}}} 
\otimes v_1^{m-k-j}v_2^{j}\\[3pt]
&=
\sum_{p=0}^{m}\sum_{k+j-r=p}
(-1)^r
\frac{\binom{m-k}{j}\binom{k}{k-r}}{\binom{m}{p}}
\frac{\partial^k}{\partial x_1^{k-r} \partial x_2^r}
\otimes
\widetilde{\frac{\partial^m}{\partial u_1^{m-p}\partial u_2^{p}}} 
\otimes v_1^{m-k-j}v_2^{j}.
\end{align*}
For 
\begin{equation*}
f(x_1, x_2) 
:= \sum_{p=0}^m f_p(x_1,x_2) 
\otimes u_1^{m-p}u_2^p \in C^\infty(\R^2)\otimes \C^m[u_1, u_2], 
\end{equation*}
we have
\begin{equation*}
\Cal{D}^{m-k}_{m,k}f(x_1, x_2) = 
\sum_{p=0}^{m}\sum_{k+j-r=p}
(-1)^r
\frac{\binom{m-k}{j}\binom{k}{k-r}}{\binom{m}{p}}
\frac{\partial^kf_p}{\partial x_1^{k-r} \partial x_2^r}(x_1,x_2)
\otimes v_1^{m-k-j}v_2^{j}.\\[3pt]
\end{equation*}

\begin{thm}\label{thm:IDO2}
The following holds.
\begin{equation*}
\Diff_G(I(\poly_2^m, \lambda)^\ga, I(\poly_2^\ell,  \nu)^\beta)
=
\begin{cases}
\C\id & \text{if $(\ga, m, \lambda) = (\gb, \ell, \nu)$},\\[3pt]
\C\Cal{D}_{m,\ell-m}^{\ell} &\text{if $(\ga, \gb; m,\ell ; \lambda, \nu) \in \gL_1$},\\[3pt]
\C\Cal{D}_{m,m-\ell}^{\ell} &\text{if $(\ga, \gb; m,\ell ; \lambda, \nu) \in \gL_2$},\\
\{0\} &\text{otherwise}.
\end{cases}
\end{equation*}
\end{thm}

\begin{rem}\label{rem:diff}
The intertwining differential operators $\Cal{D}^{\ell}_{0,\ell}$ are constructed in 
\cite{KuOr25a}. In the cited paper, $\Cal{D}^{\ell}_{0,\ell}$ are denoted by $\D_\ell$.
(See also \cite[Remark 7.2]{Kubo24+}.)
\end{rem}

%\vskip 0.1in

%%%%%%%%%%%%%%%%%%%%%%%%%%%%%%%%%%%%%%%%%%%%%%%
\subsection{Explicit formulas of $\varphi$}
\label{sec:main2b}
Next we consider the explicit formulas of $(\fg, P)$-homomorphisms
\begin{equation*}
\varphi \in \Hom_{\fg,P}(M_\fp(\sym_2^{\ell},-\nu)^{\gb},
M_\fp(\sym_2^m, -\lambda)^\ga)
\end{equation*}
via the identifications
\begin{align}
\Hom_{\fg,P}(M_\fp(\sym_2^{\ell},-\nu)^{\gb},
M_\fp(\sym_2^m, -\lambda)^\ga)
&\simeq 
\Hom_P(S^\ell(\C^2)_\beta \boxtimes \C_{-\nu} , 
 M_\fp(\sym_2^m, -\lambda)^\ga) \nonumber \\[3pt]
&\subset \Hom_\C(S^\ell(\C^2), S(\fn_-)\otimes S^m(\C^2)) \nonumber \\[3pt]
&\simeq S(\fn_-)\otimes S^m(\C^2)\otimes \Pol^\ell(\C^2).\label{eqn:gP52}
\end{align}

As in Section \ref{sec:main2a}, we regard
\begin{alignat*}{2}
&\Pol^\ell(\C^2) \simeq
\C^\ell[v_1, v_2],
\quad
&&S^\ell(\C^2) \simeq 
\C^{\ell}\left[\tfrac{\partial}{\partial v_1}, \tfrac{\partial}{\partial v_2}\right],\\[3pt]
&S(\fn_-)=\C[N^-_1, N^-_2],
\quad
&&S^m(\C^2) \simeq 
\C^m\left[\tfrac{\partial}{\partial u_1}, \tfrac{\partial}{\partial u_2}\right].
\end{alignat*}
We then identify
the map
\begin{equation*}
\varphi \colon S^\ell(\C^2) \To S(\fn_-) \otimes S^m(\C^2)
\end{equation*}
with
\begin{equation*}
\varphi \colon 
\C^{\ell}\left[\tfrac{\partial}{\partial v_1}, \tfrac{\partial}{\partial v_2}\right]
\To
\C[N_1^-, N_2^-]
\otimes 
\C^m\left[\tfrac{\partial}{\partial u_1}, \tfrac{\partial}{\partial u_2}\right],
\end{equation*}
which is, via \eqref{eqn:gP52}, considered as
\begin{equation*}
\varphi \in \C[N^-_1, N^-_2]
\otimes \C^m\left[\tfrac{\partial}{\partial u_1}, \tfrac{\partial}{\partial u_2}\right]
\otimes \C^\ell[v_1, v_2].
\end{equation*}
For $v_1^{\ell-p}v_2^p \in \C^\ell[v_1,v_2]$,
we write
\begin{equation*}
\widetilde{(v_1^{\ell-p}v_2^p)}:=\frac{1}{(\ell-p)! p!}\cdot v_1^{\ell-p}v_2^p.
\end{equation*}

\vskip 0.1in

%%%%%%%%%%%%%%%%%%%%%%%%%%%%%%%%%%%%%%%%%%
For $(\ga, \gb; m,\ell; \lambda,\nu) \in \gL_1$ with $k:=\ell-m$,
we define 
\begin{equation*}
\varphi_{m,k}^{m+k} \equiv \varphi_{m,\ell-m}^{\ell} \colon 
\C^{m+k}\left[\tfrac{\partial}{\partial v_1}, \tfrac{\partial}{\partial v_2}\right]
\To
\C^k[N_1^-, N_2^-]
\otimes \C^m\left[\tfrac{\partial}{\partial u_1}, \tfrac{\partial}{\partial u_2}\right]
\end{equation*}
by
\begin{equation}\label{eqn:gP1}
\varphi_{m,k}^{m+k} :=
\frac{1}{(m+k)!}
\left(\tfrac{\partial}{\partial u_1}\otimes v_1 + \tfrac{\partial}{\partial u_2}\otimes v_2\right)^m
\left(N_1^-\otimes v_1 + N_2^-\otimes v_2\right)^k,
\end{equation}
which amounts to
\begin{align*}
\varphi_{m,k}^{m+k}
&=
\sum_{j=0}^m\sum_{r=0}^k
\frac{\binom{m}{j}\binom{k}{r}}{\binom{m+k}{j+r}}
(N_1^-)^{k-r}(N_2^-)^r
\otimes \frac{\partial^m}{\partial u_1^{m-j}\partial u_2^j}
\otimes \widetilde{\left(v_1^{m+k-j-r}v_2^{j+r}\right)}\\[3pt]
&=\sum_{p=0}^{m+k}\sum_{j+r=p}
\frac{\binom{m}{j}\binom{k}{r}}{\binom{m+k}{p}}
(N_1^-)^{k-r}(N_2^-)^r
\otimes \frac{\partial^m}{\partial u_1^{m-j}\partial u_2^j}
\otimes \widetilde{\left(v_1^{m+k-p}v_2^{p}\right)}.
\end{align*}
For 
\begin{equation*}
v:= \sum_{p=0}^{m+k} a_{p} 
\frac{\partial^{m+k}}{\partial v_1^{m+k-p} \partial v_2^{p}}
\in
\C^{m+k}\left[\tfrac{\partial}{\partial v_1}, \tfrac{\partial}{\partial v_2}\right],
\end{equation*}
we have 
\begin{equation*}
\varphi_{m,k}^{m+k}(v) = 
\sum_{p=0}^{m+k}\sum_{j+r=p}
a_p
\frac{\binom{m}{j}\binom{k}{r}}{\binom{m+k}{p}}
(N_1^-)^{k-r}(N_2^-)^r
\otimes \frac{\partial^m}{\partial u_1^{m-j}\partial u_2^j}.\\[3pt]
\end{equation*}

%%%%%%%%%%%%%%%%%%%%%%%%%%%%%%%%%%%%%%%%%%
For $(\ga, \gb; m, \ell; \lambda,\nu) \in \gL_2$ with $k:=m-\ell$,
we define
\begin{equation*}
\varphi_{m,k}^{m-k}
\equiv
\varphi_{m,m-\ell}^{\ell} 
\colon 
\C^{m-k}\left[\tfrac{\partial}{\partial v_1}, \tfrac{\partial}{\partial v_2}\right]
\To
\C^k[N_1^-, N_2^-]
\otimes \C^m\left[\tfrac{\partial}{\partial u_1}, \tfrac{\partial}{\partial u_2}\right]
\end{equation*}
by
\begin{equation}\label{eqn:gP2}
\varphi_{m,k}^{m-k} :=
\frac{1}{(m-k)!}
\left(\tfrac{\partial}{\partial u_1}\otimes v_1 + \tfrac{\partial}{\partial u_2}\otimes v_2\right)^{m-k}
\left(N_1^-\otimes\tfrac{\partial}{\partial u_2} - N_2^-\otimes \tfrac{\partial}{\partial u_1}\right)^k,
\end{equation}
which amounts to
\begin{equation*}
\varphi_{m,k}^{m-k}
=
\sum_{j=0}^{m-k}\sum_{r=0}^k
(-1)^r
\binom{k}{r}
(N_1^-)^{k-r}(N_2^-)^r
\otimes \frac{\partial^{m}}{\partial u_1^{m-k-j+r}\partial u_2^{k+j-r}}
\otimes \widetilde{\left(v_1^{m-k-j}v_2^{j}\right)}.
\end{equation*}
For 
\begin{equation*}
v:= \sum_{j=0}^{m-k} a_{j} 
\frac{\partial^{m-k}}{\partial v_1^{m-k-j} \partial v_2^{j}}
\in
\C^{m-k}\left[\tfrac{\partial}{\partial v_1}, \tfrac{\partial}{\partial v_2}\right],
\end{equation*}
we have 
\begin{equation*}
\varphi_{m,k}^{m-k}(v)
=
\sum_{j=0}^{m-k}\sum_{r=0}^k
a_j(-1)^r 
\binom{k}{r}
(N_1^-)^{k-r}(N_2^-)^r
\otimes \frac{\partial^m}{\partial u_1^{m-k-j+r}\partial u_2^{k+j-r}}.\\[3pt]
\end{equation*}

\begin{thm}\label{thm:gP2}
The following holds.
\begin{equation*}
\Hom_{\fg,P}(M_\fp(\sym_2^{\ell},-\nu)^{\gb},
M_\fp(\sym_2^m, -\lambda)^\ga)
=
\begin{cases}
\C\id & \text{if $(\ga, m, \lambda) = (\gb, \ell, \nu)$},\\[3pt]
\C\varphi_{m,\ell-m}^{\ell} &\text{if $(\ga, \gb; m,\ell ; \lambda, \nu) \in \gL_1$},\\[3pt]
\C\varphi_{m,m-\ell}^{\ell} &\text{if $(\ga, \gb; m,\ell ; \lambda, \nu) \in \gL_2$},\\
\{0\} &\text{otherwise}.
\end{cases}
\end{equation*}
\end{thm}

%%%%%%%%%%%%%%%%%%%%%%%%%%%%%%%%%%%%%%%%%%%%%%%
\subsection{Proofs of the explicit formulas}
\label{sec:proof2}
We show the explicit formulas of 
\eqref{eqn:IDO1},
\eqref{eqn:IDO2},
\eqref{eqn:gP1},
and \eqref{eqn:gP2},
uniformly.
To the end,  we first construct
\begin{align*}
0\neq \psi
&\in \Hom_{SL(2,\R)}(S^{m+k-2d}(\C^2), S^k(\C^2)\otimes S^m(\C^2))\\[3pt]
&\simeq \left(S^k(\C^2) \otimes S^m(\C^2) \otimes 
\Pol^{m+k-2d}(\C^2)\right)^{SL(2,\R)}
\end{align*}
for $d \in \{0, \ldots, \min(k,m)\}$.

Let $\eps_j \in (\C^2)^\vee$ denote 
the dual basis of $e_j\in \C^2$ for $j=1,2$, namely, $\eps_j(e_i) = \delta_{i,j}$.
We regard $\eps_j$ as the dual basis of $N_j^-$ as well.
Write
\begin{equation*}
S^k(\C^2) = \C^k[N_1^-, N_2^-],
\quad
S^m(\C^2)=\C^m[e_1, e_2],
\quad
\Pol^{m+k-2d}(\C^2)=\C^{m+k-2d}[\eps_1,\eps_2].
\end{equation*}
Then we wish to construct
\begin{equation*}
0\neq \psi \in 
\left(\C^k[N_1^-, N_2^-] \otimes \C^m[e_1, e_2] \otimes 
\C^{m+k-2d}[\eps_1, \eps_2]\right)^{SL(2,\R)}.\\[3pt]
\end{equation*}
As
$\dim 
\left(\C^k[N_1^-, N_2^-] \otimes \C^m[e_1, e_2] \otimes 
\C^{m+k-2d}[\eps_1, \eps_2]\right)^{SL(2,\R)}=1$ (see \eqref{eqn:CG}),
we have 
\begin{equation*}
\left(\C^k[N_1^-, N_2^-] \otimes \C^m[e_1, e_2] \otimes 
\C^{m+k-2d}[\eps_1, \eps_2]\right)^{SL(2,\R)}
=\C\psi
\end{equation*}
for some $\psi \neq 0$.

We write
\begin{equation*}
N^-:=(N_1^-, N_2^-),
\quad
e:=(e_1, e_2),
\quad
\eps:=(\eps_1, \eps_2).
\end{equation*}
Then define 
\begin{equation*}
\psi^{m+k-2d}_{m,k}(N^-, e, \eps)
\in
\C^k[N_1^-, N_2^-] \otimes \C^m[e_1, e_2] \otimes 
\C^{m+k-2d}[\eps_1, \eps_2]
\end{equation*}
by
\begin{equation}\label{eqn:psi}
\psi^{m+k-2d}_{m,k}(N^-, e, \eps)
:=
(N_1^-\otimes \eps_1 + N_2^-\otimes \eps_2)^{k-d}
(e_1\otimes \eps_1 + e_2\otimes \eps_2)^{m-d}
(N_1^-\otimes e_2 - N_2^-\otimes e_1)^{d}.\\[3pt]
\end{equation}

\begin{prop}\label{prop:psi}
For $d \in \{0, \ldots, \min(k,m)\}$,
we have 
\begin{equation*}
\left(\C^k[N_1^-, N_2^-] \otimes \C^m[e_1, e_2] \otimes 
\C^{m+k-2d}[\eps_1, \eps_2]\right)^{SL(2,\R)}
=\C\psi^{m+k-2d}_{m,k}(N^-, e, \eps).
\end{equation*}
\end{prop}

\begin{proof}
It is clear from \eqref{eqn:psi} that
\begin{equation*}
\psi^{m+k-2d}_{m,k}(N^-, e, \eps)
\in 
\C^k[N_1^-, N_2^-] \otimes \C^m[e_1, e_2] \otimes 
\C^{m+k-2d}[\eps_1, \eps_2].
\end{equation*}
Thus, it suffices to show that 
$\psi^{m+k-2d}_{m,k}(N^-, e, \eps)$ is $SL(2,\R)$-invariant.

Observe that as 
\begin{equation*}
\left(\C^2\otimes (\C^2)^\vee\right)^{SL(2,\R)}
\simeq
\Hom_{SL(2,\R)}(\C^2, \C^2)
=\C \id,
\end{equation*}
we have 
\begin{equation*}
e_1 \otimes \eps_1 + e_2 \otimes \eps_2 \in 
\left(\C^2\otimes (\C^2)^\vee\right)^{SL(2,\R)}.
\end{equation*}
Therefore,
\begin{equation}\label{eqn:psi1}
(e_1 \otimes \eps_1 + e_2 \otimes \eps_2)^n \in 
\left(\C^n[e_1, e_2]\otimes \C^n[\eps_1,\eps_2]\right)^{SL(2,\R)}.\\[3pt]
\end{equation}

On the other hand, it follows from the $SL(2,\R)$-invariance of the determinant
$
\det\begin{pmatrix} x & w \\ y & z \end{pmatrix}
=xz-wy
$
that we have 
\begin{equation*}
N_1^- \otimes e_2-N_2^-\otimes e_1
\in 
\left(\fn_-\otimes \C^2\right)^{SL(2,\R)},
\end{equation*}
which yields 
\begin{equation}\label{eqn:psi2}
(N_1^- \otimes e_2-N_2^-\otimes e_1)^n
\in \left(\C^n[N_1^-,N_2^-]\otimes \C^n[e_1,e_2]\right)^{SL(2,\R)}.\\[3pt]
\end{equation}

Now the $SL(2,\R)$-invariance of $\psi^{m+k-2d}_{m,k}(N^-, e, \eps)$ 
follows from \eqref{eqn:psi1} and \eqref{eqn:psi2}.
\end{proof}

We write
\begin{equation*}
\partial_u:=\left(\tfrac{\partial}{\partial u_1}, \tfrac{\partial}{\partial u_2}\right)
\quad \text{and} \quad
v:=(v_1, v_2).
\end{equation*}
Then, under the identifications $M_0 \simeq SL(2,\R)$ and
\begin{alignat*}{2}
\C^m[e_1, e_2] &\simeq S^m(\C^2)&&\simeq \C^m\left[\tfrac{\partial}{\partial u_1}, \tfrac{\partial}{\partial u_2}\right],\\[5pt]
\C^{m+k-2d}[\eps_1, \eps_2]&\simeq \Pol^{m+k-2d}(\C^2)&&\simeq 
\C^{m+k-2d}[v_1, v_2],
\end{alignat*}
Proposition \ref{prop:psi} shows that
\begin{equation*}
\left(\C^k[N_1^-, N_2^-] \otimes 
\C^m\left[\tfrac{\partial}{\partial u_1}, \tfrac{\partial}{\partial u_2}\right]
\otimes 
\C^{m+k-2d}[v_1, v_2]\right)^{M_0}
=\C\psi^{m+k-2d}_{m,k}(N^-,\partial_u,v).\\[3pt]
\end{equation*}
In particular, for $d=0$, we have 
\begin{align*}
\C
\psi^{m+k}_{m,k}(N^-,\partial_u,v)
&=
\left(\C^k[N_1^-, N_2^-] \otimes 
\C^m\left[\tfrac{\partial}{\partial u_1}, \tfrac{\partial}{\partial u_2}\right]
\otimes 
\C^{m+k}[v_1, v_2]\right)^{M_0}
\end{align*}
with
\begin{equation*}
\psi^{m+k}_{m,k}(N^-,\partial_u,v)
=
\left(\tfrac{\partial}{\partial u_1}\otimes v_1 + \tfrac{\partial}{\partial u_2}\otimes v_2\right)^m
\left(N_1^-\otimes v_1 + N_2^-\otimes v_2\right)^k.
\end{equation*}

\vskip 0.1in

For $d=k$ with $m\geq k$, we have 
\begin{align*}
\C
\psi^{m-k}_{m,k}(N^-,\partial_u,v)
&=
\left(\C^k[N_1^-, N_2^-] \otimes 
\C^m\left[\tfrac{\partial}{\partial u_1}, \tfrac{\partial}{\partial u_2}\right]
\otimes 
\C^{m-k}[v_1, v_2]\right)^{M_0}
\end{align*}
with
\begin{equation*}
\psi^{m-k}_{m,k}(N^-,\partial_u,v)
=
\left(\tfrac{\partial}{\partial u_1}\otimes v_1 + \tfrac{\partial}{\partial u_2}\otimes v_2\right)^{m-k}
\left(N_1^-\otimes\tfrac{\partial}{\partial u_2} - N_2^-\otimes \tfrac{\partial}{\partial u_1}\right)^k.\\[3pt]
\end{equation*}

Now we are ready to prove Theorems \ref{thm:IDO2} and \ref{thm:gP2}.

%%%%%%%%%%%%%%%%%%%%%%%%%%%%%%%%%%%%%%%%%%%%%%%
\begin{proof}[Proofs of Theorems \ref{thm:IDO2} and \ref{thm:gP2}]
By \eqref{eqn:gP52}, we have 
\begin{equation*}
\Hom_P(S^\ell(\C^2)_\beta \boxtimes \C_{-\nu} , 
 M_\fp(\sym_2^m, -\lambda)^\ga)
\subset
\C[N^-_1, N^-_2]
\otimes \C^m\left[\tfrac{\partial}{\partial u_1}, \tfrac{\partial}{\partial u_2}\right]
\otimes \C^\ell[v_1, v_2].
\end{equation*}
Further,
Theorem \ref{thm:gP1} and Observation \ref{obs:41} show that,
for $(\ga, \gb; m,\ell; \lambda,\nu) \in \gL$ with $k:=|m-\ell|$, we have 
\begin{equation}
\Hom_P(S^\ell(\C^2)_\beta \boxtimes \C_{-\nu} , 
 M_\fp(\sym_2^m, -\lambda)^\ga)
 =
\left(\C^k[N^-_1, N^-_2]
\otimes \C^m\left[\tfrac{\partial}{\partial u_1}, \tfrac{\partial}{\partial u_2}\right]
\otimes \C^\ell[v_1, v_2]\right)^{M_0}. \label{eqn:P53}\\[3pt]
\end{equation}
Since
\begin{equation*}
\varphi^{m+k}_{m,k}=\frac{1}{(m+k)!}\psi^{m+k}_{m,k}(N^-,\partial_u,v)
\quad
\text{and}
\quad \varphi^{m-k}_{m,k}=\frac{1}{(m-k)!}\psi^{m-k}_{m,k}(N^-,\partial_u,v),
\end{equation*}
the arguments preceding to this proof conclude Theorem \ref{thm:gP2}.

Theorem \ref{thm:IDO2} just follows from 
Theorems \ref{thm:duality} 
and \ref{thm:gP2}. Indeed, by these theorems, we have 
\begin{align*}
\EuScript{D}
\colon
\Hom_P(S^\ell(\C^2)_\beta \boxtimes \C_{-\nu} , 
 M_\fp(\sym_2^m, -\lambda)^\ga)
&\stackrel{\sim}{\to}
\Diff_G(I(\poly_2^m, \lambda)^\ga, I(\poly_2^\ell,  \nu)^\beta),\\[3pt]
\psi^{\ell}_{m,|m-\ell|}(N^-,\partial_u,v) 
&\mapsto
\psi^{\ell}_{m,|m-\ell|}(dR(N^-),\partial_u,v).
\end{align*}
As $dR(N_j^-) = \tfrac{\partial}{\partial x_j}$ for $j=1,2$,
the differential operators $\psi^{\ell}_{m,|m-\ell|}(dR(N^-),\partial_u,v)$ are given by
\begin{equation*}
\psi^{\ell}_{m,|m-\ell|}(dR(N^-),\partial_u,v)=
\psi^{m+k}_{m,k}(\partial_x,\partial_u,v),
\end{equation*}
where 
$\partial_x:=\left(\tfrac{\partial}{\partial x_1}, \tfrac{\partial}{\partial x_2}\right)$, 
Since 
\begin{equation*}
\Cal{D}^{m+k}_{m,k}=\frac{1}{m!}\psi^{m+k}_{m,k}(\partial_x,\partial_u,v)
\quad
\text{and}
\quad \Cal{D}^{m-k}_{m,k}=\frac{1}{m!}\psi^{m-k}_{m,k}(\partial_x,\partial_u,v),
\end{equation*}
the desired assertion holds.
%Now Theorem \ref{thm:IDO2} follows from 
\end{proof}

%%%%%%%%%%%%%%%%%%%%%%%%%%%%%%%%%%%%%%%%%%
\subsection{Cartan operators and PRV operators}
\label{sec:comp}
In Theorem \ref{thm:IDO2}, we constructed two families
$\Cal{D}^{m+k}_{m,k}$ and $\Cal{D}^{m-k}_{m,k}$ of differential operators. 
To end this section, we study a relationship between
$\Cal{D}^{m+k}_{m,k}$ and $\Cal{D}^{m-k}_{m,k}$.
In order to distinguish between them,
we write
\begin{equation}\label{eqn:CPdef}
\Cal{C}^{m+k}_{m,k}:=\Cal{D}^{m+k}_{m,k}
\quad
\text{and}
\quad
\Cal{P}^{m-k}_{m,k}:=\Cal{D}^{m-k}_{m,k}.
\end{equation}

As remarked in Remark \ref{rem:CPRV}, the operators 
$\Cal{C}^{m+k}_{m,k}$ and $\Cal{P}^{m-k}_{m,k}$ are obtained 
from the Cartan component and PRV component of 
the tensor product of two irreducible finite-dimensional representations of $SL(2,\R)$,
respectively. We then call $\Cal{C}^{m+k}_{m,k}$ and $\Cal{P}^{m-k}_{m,k}$
\emph{Cartan operators} and \emph{PRV operators}, respectively.
We remark that $\Cal{C}^{k}_{0,k}$ are also regarded as PRV operators
as there is no tensor product decomposition in the case (see Remark \ref{rem:CPRV}).

To simplify notation, we write
\begin{equation}\label{eqn:Ind}
I(m,\lambda)^\ga:=I(\poly^m_2,\lambda)^\ga.
\end{equation}
Then, by Theorem \ref{thm:IDO2}, we have
\begin{equation}\label{eqn:Cartan}
\Cal{C}^{m+k+1}_{m,k+1}\colon I(m, -\tfrac{1}{2}(m+2k))^\ga
\To
I(m+k+1, \tfrac{1}{2}(3-m+k))^{\ga+k+1}
\end{equation}
and
\begin{equation}\label{eqn:PRV}
\Cal{P}^{m-k-1}_{m,k+1}\colon I(m, \tfrac{1}{2}(2+m-2k))^\ga
\To
I(m-k-1, \tfrac{1}{2}(5+m+k))^{\ga+k+1},
\end{equation}
where, for $\ga \in \{\pm\}\equiv \{\pm 1\}$ 
and $n \in \Z_{\geq 0}$, we define
$\ga + n \in \{\pm\}$ as in \eqref{eqn:ak}.
We denote the kernel and image of
 $\D \in \{\Cal{C}^{m+k+1}_{m,k+1}, \Cal{P}^{m-k-1}_{m,k+1}\}$
in \eqref{eqn:Cartan} 
and \eqref{eqn:PRV}  by 
\begin{equation*}
\Ker(\D)^\ga
\quad
\text{and}
\quad
\Im(\D)^{\ga+k+1}.
\end{equation*}

We put
\begin{equation*}
\text{Cartan}:=\{\Cal{C}^{m+k+1}_{m,k+1}: m, k \in \Z_{\geq 0}\}
\quad
\text{and}
\quad
\text{PRV}:=\{\Cal{P}^{m-k-1}_{m,k+1}: m, k, m-k-1 \in \Z_{\geq 0}\}.
\end{equation*}

\begin{thm}\label{thm:CPRV}
There exists a bijection
\begin{equation*}
\Theta\colon \Cartan \To \PRV
\end{equation*}
such that
\begin{equation*}
\Im\left(\Cal{C}^{m+k+1}_{m,k+1}\right)^{\ga+k+1} 
\subset 
\Ker\left(\Theta(\Cal{C}^{m+k+1}_{m,k+1})\right)^{\ga+k+1}.
\end{equation*}
\end{thm}

\begin{proof}
Put 
\begin{equation*}
A:=
\begin{pmatrix}
1 & 1 & 1 \\
1 & 0 & 0\\
0 & 0& 1
\end{pmatrix}
\in GL(2,\R) \ltimes \R^2.
\end{equation*}
Then, for $m,k \in \Z_{\geq 0}$, we set
\begin{equation}\label{eqn:mk}
(\wm, \wk):=
(m+k+1, m)=
A
(m,k, 1)^t,
\end{equation}
where the space $\R^2$ is realized as 
an affine subspace of $\R^3$ under the identification
$(a, b)^t \simeq (a, b, 1)^t$.
Then define 
\begin{equation*}
\Theta(\Cal{C}^{m+k+1}_{m,k+1}):=
\Cal{P}^{\wm-\wk-1}_{\wm,\wk+1}.
\end{equation*}
We show that $\Theta$ satisfies the desired properties.
\vskip 0.1in

(1) Well-definedness: As  $\wm=m+k+1\geq m+1 =\wk+1$,
the PRV operator 
\begin{equation}\label{eqn:P}
\Cal{P}^{\wm-\wk-1}_{\wm,\wk+1}\colon 
I(\wm, \tfrac{1}{2}(2+\wm-2\wk))^{\ga}
\To
I(\wm-\wk-1, \tfrac{1}{2}(5+\wm+\wk))^{\ga+\wk+1}
\end{equation}
is well-defined. Thus, $\Theta(\Cal{C}^{m+k+1}_{m,k+1}) \in \PRV$.
\vskip 0.1in

(2) Inverse: 
Observe that 
the inverse $A^{-1}$ of the matrix $A$  is 
\begin{equation*}
A^{-1}=
\begin{pmatrix}
0 & 1 & 0 \\
1 & -1 & -1\\
0 & 0& 1
\end{pmatrix}.
\end{equation*}
Then, for $m,k \in \Z_{\geq 0}$ with $m-k-1 \in \Z_{\geq 0}$,
we set
\begin{equation}\label{eqn:mk2}
(\widehat{m}, \widehat{k}):=(k, m-k-1)
=A^{-1}(m,k,1)^t.
\end{equation}
By the choice of $m,k$, clearly we have $\widehat{m}, \widehat{k} \in \Z_{\geq 0}$.
Thus the Cartan operator 
\begin{equation*}
\Cal{C}^{\widehat{m}+\widehat{k}+1}_{\widehat{m},\widehat{k}+1}\colon 
I(\widehat{m}, -\tfrac{1}{2}(\widehat{m}+2\widehat{k}))^{\ga}
\To
I(\widehat{m}+\widehat{k}+1, 
\tfrac{1}{2}(3-\widehat{m}+\widehat{k}))^{\ga+\widehat{k}+1}
\end{equation*}
is well-defined. We then define 
$\Theta^{-1}(\Cal{P}^{m-k-1}_{m,k+1}):=
\Cal{C}^{\widehat{m}+\widehat{k}+1}_{\widehat{m},\widehat{k}+1}$.
By \eqref{eqn:mk} and \eqref{eqn:mk2},
the maps  $\Theta$ and $\Theta^{-1}$ are clearly inverses to each other.
\vskip 0.1in

(3) $\Im \subset \Ker$: 
By \eqref{eqn:mk}, we have
\begin{align*}
\Cal{P}^{\wm-\wk-1}_{\wm,\wk+1}&=\Cal{P}^{k}_{m+k+1,m+1},\\[3pt]
I(\wm, \tfrac{1}{2}(2+\wm-2\wk))^{\ga+k+1} 
&= I(m+k+1, \tfrac{1}{2}(3-m+k))^{\ga+k+1}, \\[3pt]
I(\wm-\wk-1, \tfrac{1}{2}(5+\wm+\wk))^{\ga+\wk+1} &= 
I(k, \tfrac{1}{2}(6+2m+k))^{\ga+m+k}.
\end{align*}
Then \eqref{eqn:P} reads
\begin{equation}\label{eqn:PRV2}
\Cal{P}^{k}_{m+k+1,m+1}\colon 
I(m+k+1, \tfrac{1}{2}(3-m+k))^{\ga+k+1}
\To
I(k, \tfrac{1}{2}(6+2m+k))^{\ga+m+k}. 
\end{equation}
By \eqref{eqn:Cartan} and \eqref{eqn:PRV}, we have
\begin{equation}\label{eqn:Cartan3}
\begin{aligned}
I(m, -\tfrac{1}{2}(m+2k))^\ga
&\stackrel{\Cal{C}^{m+k+1}_{m,k+1}}{\To}
I(m+k+1, \tfrac{1}{2}(3-m+k))^{\ga+k+1}\\[5pt]
&\stackrel{\Cal{P}^{k}_{m+k+1,m+1}}{\To}
I(k, \tfrac{1}{2}(6+2m+k))^{\ga+m+k}.
\end{aligned}
\end{equation}
Clearly, the composition $\Cal{P}^{k}_{m+k+1,m+1}\circ \Cal{C}^{m+k+1}_{m,k+1}$ is 
an intertwining differential operator. However,  Theorem \ref{thm:IDO2} shows that
\begin{equation*}
\Diff_G(I(m, -\tfrac{1}{2}(m+2k))^\ga, 
I(k, \tfrac{1}{2}(6+2m+k))^{\ga+m+k}) = \{0\}.
\end{equation*}
Therefore, we have 
\begin{equation*}
\Im\left(\Cal{C}^{m+k+1}_{m,k+1}\right)^{\ga+k+1} 
\subset 
\Ker\left(\Cal{P}^{k}_{m+k+1,m+1}\right)^{\ga+k+1}.
\end{equation*}
This proves the proposition.
\end{proof}

\begin{rem}
One can easily check that $(\Cal{C}^{m+k+1}_{m,k+1},
\Theta(\Cal{C}^{m+k+1}_{m,k+1}))$ is the only pair 
$(\D_1, \D_2)$ of non-identity operators $\D_1, \D_2$
such that the composition $\D_2 \circ \D_1$ is defined as an intertwining differential operator.
\end{rem}

In the next section, we shall show that 
\begin{equation}\label{eqn:CP}
\Im\left(\Cal{C}^{m+k+1}_{m,k+1}\right)^{\ga'} 
=
\Ker\left(\Cal{P}^{k}_{m+k+1,m+1}\right)^{\ga'}
\end{equation}
for appropriate $\ga'  \in \Z/2\Z$ for which the sequence \eqref{eqn:Cartan3} 
constitutes a BGG resolution.

%%%%%%%%%%%%%%%%%%%%%%%%%%%%%%%%%%%%%%%%%%%%%%%
\section{BGG resolution}\label{sec:BGG}

%%%%%%%%%%%%%%%%%%%%%%%%%%%%%%%%%%%%%%%%%%%%%%%

The aim of this section is to show that the sequence \eqref{eqn:Cartan3} 
is  an exact sequence for appropriate $\ga' \in \Z/2\Z$. As byproducts,
we determine the $G$-representations on
$\Ker(\Cal{C}^{m+k+1}_{m,k+1})^{\ga'}$ and 
$\Im(\Cal{P}^{k}_{m+k+1,m+1})^{\ga'}$,
and also show that the identity \eqref{eqn:CP} holds for such $\ga'$.
These are achieved in Theorem \ref{thm:BGGres} and Corollary \ref{cor:BGGres}.
We continue to denote our intertwining differential operators $\D$ by
$\Cal{C}^{m+k+1}_{m,k+1}$ and $\Cal{P}^{m-k-1}_{m,k+1}$ 
as in \eqref{eqn:CPdef}.

We remark that one may read off part of 
our results in this section in \cite[Prop.\ 2.1]{CGH12b} and \cite[Sect.\ 4]{EG11}.

%%%%%%%%%%%%%%%%%%%%%%%%%%%%%%%%%%%%%%%%%%%
\subsection{Main results}
\label{sec:main6}
We start with the main results of this section; the proofs of them will be 
discussed in Section \ref{sec:const}.

Let $\varpi_1$ and $\varpi_2$ be the first and second fundamental weights
of $\fg = \f{sl}(3,\C)$, respectively, that is,
\begin{equation*}
\varpi_1:=\frac{1}{3}(2,-1,-1)
\quad
\text{and}
\quad
\varpi_2:=\frac{1}{3}(1,1,-2).
\end{equation*}
We denote by $V(m\varpi_1 + k\varpi_2)$ the irreducible finite-dimensional 
representation of $\fg$ with highest weight $m\varpi_1 + k\varpi_2$.

\begin{thm}[BGG resolution]
\label{thm:BGGres}
There exists an injective $G$-intertwining operator 
\begin{equation*}
\wT\colon
V(m\varpi_1 + k \varpi_2)
\longhookrightarrow
I(m, -\tfrac{1}{2}(m+2k))^{m+k}
\end{equation*}
such that the following sequence 
is exact:
\begin{equation}\label{eqn:BGGres}
\begin{aligned}
0\To
V(m\varpi_1 + k \varpi_2)
&\stackrel{\wT}{\To}
I(m, -\tfrac{1}{2}(m+2k))^{m+k}\\
&\stackrel{\Cal{C}^{m+k+1}_{m,k+1}}{\To}
I(m+k+1, \tfrac{1}{2}(3-m+k))^{m+1}\\[5pt]
&\stackrel{\Cal{P}^{k}_{m+k+1,m+1}}{\To}
I(k, \tfrac{1}{2}(6+2m+k))^{+}
\To 0
\end{aligned}
\end{equation}
with $I(k, \tfrac{1}{2}(6+2m+k))^{+}=I(k, \tfrac{1}{2}(6+2m+k))^{2m+2}$
(see \eqref{eqn:ak}).
\end{thm}

It is known that the resolution \eqref{eqn:BGGres} is a so-called
BGG resolution for the pair
$(G, P)= (SL(3,\R), P_{1,2})$ (cf.\  \cite{CD01, CSS01, EG11}).
Theorem \ref{thm:BGGres} shows that the Cartan operators
$\Cal{C}^{m+k+1}_{m,k+1}$ and PRV operators 
$\Cal{P}^{k}_{m+k+1,m+1}$ 
are the first and second BGG operators, respectively.

\begin{rem}
As $V(0\varpi_1 + 0\varpi_2)=\C_{\triv}$ is the trivial
representation of $G$, 
if $(m,k)=(0,0)$, then
\eqref{eqn:BGGres} can be thought of as
the de Rham complex on $\R^2$: 
\begin{equation*}%\label{eqn:ext1}
0\To
\C
\To
C^\infty(\R^2) \otimes \Exterior^0(\C^2)
\stackrel{d}{\To}
C^\infty(\R^2)\otimes \Exterior^1(\C^2)
\stackrel{d}{\To}
C^\infty(\R^2)\otimes \Exterior^2(\C^2)
\To 0
\end{equation*}
with $d=\Cal{C}^{1}_{0,1}$ for the first differential 
and $d=\Cal{P}^{0}_{1,1}$ for the second differential 
(see \eqref{eqn:IDO1} and \eqref{eqn:IDO2}).
\end{rem}

\begin{rem}
As in \eqref{eqn:Ind}, write
\begin{equation*}
M_\fp(m,\lambda):=M_\fp(\sym^m_2,\lambda).
\end{equation*}
Here, as opposed to Sections \ref{sec:classification} and \ref{sec:construction}, 
we regard generalized Verma modules merely as $\fg$-modules; thus, we omit the 
parity condition $\ga \in \Z/2\Z$. Then the BGG resolution of generalized 
Verma modules corresponding to 
\eqref{eqn:BGGres} is given as 
\begin{equation}\label{eqn:BGGgeneral2}
\begin{aligned}
0
\To M_\fp(k, -\tfrac{1}{2}(6+2m+k))
&\stackrel{\varphi^{k}_{m+k+1,m+1}}{\To}
M_\fp(m+k+1, -\tfrac{1}{2}(3-m+k))\\[5pt]
&\stackrel{\varphi^{m+k+1}_{m,k+1}}{\To}
M_\fp(m, \tfrac{1}{2}(m+2k))\\[5pt]
&\stackrel{\pr}{\To}
V(k\varpi_1 + m\varpi_2)
\To 0,
\end{aligned}
\end{equation}
where $\pr$ denotes the natural projection onto its irreducible quotient
and $V(k\varpi_1 + m\varpi_2)= V(m\varpi_1+k\varpi_2)^\vee$,
the dual of $V(m\varpi_1+k\varpi_2)$.
\end{rem}

The next corollary is an immediate consequence of Theorem \ref{thm:BGGres}.

\begin{cor}\label{cor:BGGres}
The following hold.
\begin{enumerate}

\item[\emph{(1)}] $\Ker(\Cal{C}^{m+k+1}_{m,k+1})^{m+k}\simeq V(m\varpi_1+k\varpi_2)$.
\vskip 0.1in

\item[\emph{(2)}] $\Im(\Cal{C}^{m+k+1}_{m,k+1})^{m+1} 
= \Ker(\Cal{P}^{k}_{m+k+1,m+1})^{m+1}$.

\vskip 0.1in

\item[\emph{(3)}] $\Im(\Cal{P}^{k}_{m+k+1,m+1})^{+} = I(k, \tfrac{1}{2}(6+2m+k))^{+}$.

\end{enumerate}
\end{cor}

We remark that Corollary \ref{cor:BGGres} does not fully determine
the representations on the kernels and images of 
 intertwining differential operators in consideration.
 Indeed, if we change the parity condition $m+k$
 for $I(m, -\tfrac{1}{2}(m+2k))^{m+k}$ to 
 $m+k+1$, then \eqref{eqn:BGGres} 
 does not conclude anything for the representation on 
 the kernel  $\Ker(\Cal{C}^{m+k+1}_{m,k+1})^{m+k+1}$ of 
$ \Cal{C}^{m+k+1}_{m,k+1}$ for
 \begin{equation*}
\Cal{C}^{m+k+1}_{m,k+1} \colon
I(m, -\tfrac{1}{2}(m+2k))^{m+k+1}
\To
I(m+k+1, \tfrac{1}{2}(3-m+k))^{m}.
\end{equation*}
We shall show in Proposition \ref{prop:KI} that 
$\Ker(\Cal{C}^{m+k+1}_{m,k+1})^{m+k+1}=\{0\}$
by making use of the ellipticity of Cartan operators $\Cal{C}^{m+k+1}_{m,k+1}$.
 
\begin{rem}\label{rem:Ktype}
In \cite{KuOr25a},
the composition structures of parabolically induced representations
are utilized to determine the $K$-type structure of 
the $G$-representation on the kernel $\Ker(\Cal{C}^{k+1}_{0,k+1})^{\ga}$  for any $\ga$
for $SL(n,\R)$. Further, 
the kernels $\Ker(\Cal{C}^{k+1}_{0,k+1})^{\ga}$ are 
computed explicitly in the non-compact picture in \cite{Kubo24+}.
(See Remark \ref{rem:diff} for the difference of notation of differential operators
between \cite{Kubo24+, KuOr25a} and this paper.)
\end{rem}

%%%%%%%%%%%%%%%%%%%%%%%%%%%%%%%%%%%%%%%%%%%%%%%
\subsection{Construction of $\wT$}
\label{sec:const}
Remark that the construction of the desired $\wT$ in Theorem \ref{thm:BGGres}
is well known in the context of 
parabolic geometry in great generality (see, for instance, \cite{CD01, CSS01}). 
In particular, the BGG resolution in concern is studied in, for instance,
{\v C}ap--Gover--Hammerl \cite{CGH12b} and 
Eastwood--Gover \cite{EG11}.
However, in contrast to the cited papers, we need to determine the character
$\ga \in \Z/2\Z$ of $M$ explicitly. In doing so, we give a detailed account of $\wT$ 
in this section.

%%%%%%%%%%%%%%%%%%%%%%%%%%%%%%%%%%%%%%%%%%%%%%%
\subsubsection{Grading on $V(m\varpi_1+k\varpi_2)$}
\label{sec:grading}
We start with a grading on $V(m\varpi_1+k\varpi_2)$ induced by the
parabolic subgroup $P=MAN_+$. Recall from \eqref{eqn:H0} that 
\begin{equation}\label{eqn:H02}
\wH_0=\frac{1}{2}\diag(2, -1, -1).
\end{equation}
For $s \in \C$, we write $W[s]$ for the eigenspace of $\wH_0$ associated 
with eigenvalue $s$, namely,
\begin{equation*}
W[s] := \{v \in V(m\varpi_1 + k \varpi_2) : \wH_0 \cdot v = sv\},
\end{equation*}
where the dot $(\cdot)$ denotes the action of $\wH_0$ on 
$V(m\varpi_1 + k \varpi_2)$. Since $M$ commutes with $A=\exp(\R \wH_0)$, 
the space $W[s]$ is an $MA$-representation.

As $\fm\simeq \f{sl}(2,\C)$, the restriction 
$V(m\varpi_1 + k \varpi_2)\vert_{\fm}$ decomposes into a direct sum of 
$\fm$-irreducible components.
Let $W_\ell$ denote the irreducible component containing
a unique (up to scalar) lowest weight vector $v_{\ell}$ of $V(m\varpi_1 + k \varpi_2)$.
We then have  $W_\ell =\Cal{U}(\fm)v_\ell$.
Since $A$ commutes with $M$, the group $A$ acts on $W_\ell$ by scalar. 
Thus, $W_\ell \subset W[s_0]$ for some $s_0 \in \C$. Further, if 
$W_\ell \subsetneq W[s_0]$, then there would be another lowest weight vector
$v_\ell' \notin \C v_\ell$, which contradicts the irreducibility of $V(m\varpi_1+k\varpi_2)$. Therefore, we have $W_\ell=W[s_0]$.

It  follows from \eqref{eqn:nR1} that we have
\begin{equation}\label{eqn:n+}
\fn_+ \cdot W[s] \subset W[s+\tfrac{3}{2}]. 
\end{equation}
Since $V(m\varpi_1 + k \varpi_2)=\Cal{U}(\fn_+)v_{\ell}$, this yields
(non-irreducible)
$MA$-decomposition
\begin{equation*}
V(m\varpi_1 + k \varpi_2)\vert_{MA}
=\bigoplus_{j=0}^{r_0}W[s_0+\tfrac{3}{2}j]
\end{equation*}
with 
$W[s_0+\tfrac{3}{2}r_0] \neq \{0\}$ for some $r_0 \in \Z_{\geq 0}$.

\begin{lem}\label{lem:rs}
We have
\begin{equation*}
s_0=-\tfrac{1}{2}(m+2k)
\quad
\text{and}
\quad
r_0=m+k.
\end{equation*}
\end{lem}

\begin{proof}
Observe that the first and second fundamental weights 
$\varpi_1$  and $\varpi_2$ are 
\begin{equation*}
\varpi_1 = \frac{1}{3}(2, -1, -1) (=d\chi)
\quad
\text{and}
\quad
\varpi_2 = \frac{1}{3}(1, 1, -2)
\end{equation*}
in coordinates from Section \ref{sec:Step4}.
Then the highest weight $m\varpi_1 + k\varpi_2$ reads
\begin{equation*}
m\varpi_1 + k \varpi_2 = \frac{1}{3}(2m+k,-m+k, -m-2k).
\end{equation*}
As the longest Weyl group element is 
$s_{\eps_1-\eps_3}$,  the lowest weight of $V(m\varpi_1 + k \varpi_2)$ is 
\begin{equation*}
s_{\eps_1-\eps_3}(m\varpi_1 + k \varpi_2)=-(k\varpi_1+m\varpi_2)
\end{equation*}
 (see, for instance, \cite[Thm.\ 3.2.13]{GW09}).
In coordinates, we have 
\begin{equation}\label{eqn:lowest}
 -(k\varpi_1+m\varpi_2)=
 -\frac{1}{3}(m+2k, m-k, -2m-k). 
\end{equation}
It follows from \eqref{eqn:H02} that
$\varpi_1(\wH_0) = 1$ and $\varpi_2(\wH_0) = \frac{1}{2}$,
yeilding
\begin{align*}
(m\varpi_1 + k \varpi_2)(\wH_0) &= m + \frac{1}{2}k=\frac{1}{2}(2m + k),\\[3pt]
-(k\varpi_1+m\varpi_2)(\wH_0) &= -(k + \frac{1}{2}m)
=-\frac{1}{2}(m+2k).
\end{align*}

Now, to determine $s_0$, observe that
\begin{equation}\label{eqn:lw}
s_0 v_\ell = \wH_0 \cdot v_\ell = 
-(k\varpi_1+m\varpi_2)(\wH_0) v_\ell = -\frac{1}{2}(m+2k)v_\ell,
\end{equation}
which shows that
$s_0=-\frac{1}{2}(m+2k)$.

For $r_0$, observe that as
\begin{equation*}
\fn_+\cdot W[-\tfrac{1}{2}(m+2k)+\tfrac{3}{2}r_0] \subset
 W[-\tfrac{1}{2}(m+2k)+\tfrac{3}{2}(r_0+1)] =\{0\},
\end{equation*}
the component $W[-\tfrac{1}{2}(m+2k)+\tfrac{3}{2}r_0]$ contains
a unique (up to scalar) highest weight vector $v_h$.
As in \eqref{eqn:lw}, this shows that
\begin{equation*}
W[-\tfrac{1}{2}(m+2k)+\tfrac{3}{2}r_0] = W[\tfrac{1}{2}(2m + k)],
\end{equation*}
which forces
$r_0 = m+k$.
\end{proof}

By Lemma \ref{lem:rs}, we obtain
\begin{equation}\label{eqn:decomp}
V(m\varpi_1 + k \varpi_2)\vert_{MA}
=\bigoplus_{j=0}^{m+k}W[-\tfrac{1}{2}(m+2k)+\tfrac{3}{2}j].
\end{equation}
Now we claim that, as $MA$-representations, we have 
\begin{equation*}
W[-\tfrac{1}{2}(m+2k)] \simeq 
(\C_{m+k} \otimes \poly^{m}_2) \boxtimes \C_{-\tfrac{1}{2}(m+2k)}.
\end{equation*}
Toward the end, first we show its $\f{m}\oplus\f{a}$-equivalence.

\begin{prop}\label{prop:ma}
As $\f{m}\oplus\f{a}$-modules, we have 
\begin{equation*}
W[-\tfrac{1}{2}(m+2k)] \simeq 
\poly^{m}_2 \boxtimes \C_{-\tfrac{1}{2}(m+2k)}.
\end{equation*}
\end{prop}

\begin{proof}
As both $W[-\tfrac{1}{2}(m+2k)]$ and $\poly^{m}_2 \boxtimes \C_{-\tfrac{1}{2}(m+2k)}$
are irreducible $\f{m}\oplus\f{a}$-modules, it suffices to show that the lowest weights
of these two modules are the same.
The highest weight of $\poly^{m}_2$ is
$m\omega_1$ with $\omega_1=\frac{1}{2}(0, 1, -1)$; in particular, its lowest weight is
$-m\omega_1$. Thus, the lowest weight of $\poly^{m}_2\boxtimes  \C_{-\tfrac{1}{2}(m+2k)}$ 
is
\begin{align*}
-m\omega_1 -\tfrac{1}{2}(m+2k)d\chi
&=\frac{1}{2}(0, -m, m) - \frac{1}{6}(m+2k)(2,-1,-1)\\
&=-\frac{1}{3}(m+2k, m-k, -2m-k).
\end{align*}
On the other hand, it follows from \eqref{eqn:lowest} that the lowest weight of 
$W[-\tfrac{1}{2}(m+2k)]$, which is the lowest weight of $V(m\varpi_1 +k\varpi_2)$,
is
\begin{equation*}
 -(k\varpi_1+m\varpi_2)=
 -\frac{1}{3}(m+2k, m-k, -2m-k). 
\end{equation*}
Now the claim follows.
\end{proof}

\begin{prop}\label{prop:MA}
As $MA$-modules, we have 
\begin{equation}\label{eqn:wp}
W[-\tfrac{1}{2}(m+2k)] \simeq 
(\C_{m+k} \otimes \poly^{m}_2) \boxtimes \C_{-\tfrac{1}{2}(m+2k)}.
\end{equation}
\end{prop}

\begin{proof}
We wish to determine the character of $M\simeq SL(2,\R)^{\pm}$ on 
$W[-\tfrac{1}{2}(m+2k)]$. Let $v_\ell$ be a lowest weight vector
of $W[-\tfrac{1}{2}(m+2k)]$. Then $W[-\tfrac{1}{2}(m+2k)] = \Cal{U}(\fm)v_\ell$.
It follows from a direct computation that
\begin{equation*}
(M, \Ad, \fm) \simeq (SL^{\pm}(2,\R), \triv\otimes \Ad,\f{sl}(2,\C)).
\end{equation*}
Thus, to determine the character of $M$ on $W[-\tfrac{1}{2}(m+2k)]$, it suffices to 
observe the character of $M$ on $v_\ell$.

Let $B=TN^+_\C$ be the Borel subgroup of $G_\C:=SL(3,\C)$ consisting of 
upper triangular matrices with $T$ the group of diagonal matrices
and $N_\C$ the group of upper triangular matrices with $1$ on the diagonal.
As in \cite[Exer.\ 7.34]{Sepanski07},
we write $\xi_{(a,b)}\colon T \to \C^\times$ for the character of $T$
associated with $(a,b):=a\varpi_1+b\varpi_2$ for $a, b \in \Z$.
Then, by the Borel--Weil theorem and Weyl's unitary trick, the 
irreducible representation $V(m\varpi_1+k\varpi_2)$ of $G$ can be realized as
the space $\Cal{O}(G_\C, \C_{-(k,m)})^B$ of $B$-invariant 
holomorphic functions on $G_\C$, namely,
\begin{equation*}
\Cal{O}(G_\C, \C_{-(k,m)})^B
:=
\{f \in \Cal{O}(G_\C) : 
f(gtn)=\xi_{-(k,m)}(t^{-1})f(g)\; \text{for $tn \in TN^+_\C$}\},
\end{equation*}
where $G$ acts on $\Cal{O}(G_\C, \C_{-(k,m)})^B$ by left translation, i.e.,
$L(g_0)f(g):=f(g_0^{-1}g)$.
Then put
\begin{equation*}
f_\ell(g):=({\det}_1g)^{k}({\det}_2g)^{m},
\end{equation*}
where ${\det}_r(g_{i,j}):={\det}_{i,j\leq r}(g_{i,j})$.
One can easily check that $f_\ell(g)$ is a lowest weight vector of
$\Cal{O}(G_\C, \C_{-(k,m)})^B$.
A direct computation shows that, for $\widetilde{h}:=
\begin{pmatrix}
\det(h)^{-1} &\\
& h\\
\end{pmatrix}
\in M$ with $h \in SL^{\pm}(2,\R)$, we have
\begin{equation*}
L(\widetilde{h})f_\ell(g)=
f_\ell(\widetilde{h}^{-1}g)=\det(h)^{m+k}f_\ell(\begin{pmatrix}
1 &\\
& h^{-1}\\
\end{pmatrix}
g),
\end{equation*}
which shows that $\widetilde{h}$ acts on $\Cal{O}(G_\C, \C_{-(k,m)})^B$
as $\det(h)^{m+k} L\begin{pmatrix}
1 &\\
& h\\
\end{pmatrix}$.
Now the proposed assertion follows.
\end{proof}

%%%%%%%%%%%%%%%%%%%%%%%%%%%%%%%%%%%%%%%%%%%%%%%
\subsubsection{Construction of $\wT$}
\label{sec:wT}
Put
\begin{equation*}
W_+:=\bigoplus_{j=1}^{m+k}W[-\tfrac{1}{2}(m+2k)+\tfrac{3}{2}j],
\end{equation*}
the direct sum of the $MA$-components in \eqref{eqn:decomp} with $j> 0$, so that
\begin{equation*}
V(m\varpi_1 + k\varpi_2) = 
W[-\tfrac{1}{2}(m+2k)] \oplus W_+.
\end{equation*}
This allows to identify $MA$-representation $W[-\tfrac{1}{2}(m+2k)]$ with
\begin{equation}\label{eqn:WV}
W[-\tfrac{1}{2}(m+2k)]  \simeq V(m\varpi_1 + k \varpi_2)/W_+.
\end{equation}
It follows from \eqref{eqn:n+} that $W_+$ is an $\fn_+$-module; thus,
$W_+$ is indeed a $P$-representation. Via the identification \eqref{eqn:WV}, 
we regard $W[-\tfrac{1}{2}(m+2k)]$ as an irreducible $P$-representation.

Regarding a $G$-representation $V(m\varpi_1 + k\varpi_2)$ as a $P$-representation,
consider the induced representation from 
$V(m\varpi_1 + k\varpi_2)$:
\begin{equation*}
T_P(m\varpi_1 + k\varpi_2)
:=\Ind_{P}^G(V(m\varpi_1 + k\varpi_2)).
\end{equation*}
Then, via the $MA$-projection
\begin{equation}\label{eqn:proj1}
\proj\colon V(m\varpi_1 + k \varpi_2) \twoheadrightarrow 
V(m\varpi_1 + k \varpi_2)/W_+ \simeq W[-\tfrac{1}{2}(m+2k)]
\end{equation}
and the $MA$-module equivalence \eqref{eqn:wp}, 
one can define a $G$-intertwining map
\begin{equation}\label{eqn:proj2}
\Proj\colon 
T_P(m\varpi_1 + k \varpi_2) \To I(m, -\tfrac{1}{2}(m+2k))^{m+k},
\quad 
F\mapsto \proj \circ F.
\end{equation}

\vskip 0.1in

Now, let $\sigma\colon G \to GL(V(m\varpi_1 + k \varpi_2))$ denote the representation
of $G$ on $V(m\varpi_1 + k \varpi_2)$. For $v \in V(m\varpi_1 + k \varpi_2)$, define
\begin{align*}
\Cal{T}\colon V(m\varpi_1 + k \varpi_2) 
&\To C^\infty(G)\otimes V(m\varpi_1 + k \varpi_2)\\
v &\longmapsto f_v(g) :=\sigma(g^{-1})v.
\end{align*}
By definition, the map $\Cal{T}$ 
is an injective intertwining operator:
\begin{equation}\label{eqn:xv2}
\Cal{T}\colon V(m\varpi_1 + k\varpi_2) 
\longhookrightarrow C^\infty(G)\otimes V(m\varpi_1 + k\varpi_2).
\end{equation}
Further, as $f_v(gp)=\sigma(p^{-1})f_v(g)$ for $p \in P$, 
we have $f_v(g) \in T_P(m\varpi_1 + k \varpi_2)$. Thus,
\begin{equation}\label{eqn:xv}
\Cal{T}\colon V(m\varpi_1 + k \varpi_2) 
\longhookrightarrow T_P(m\varpi_1 + k \varpi_2).
\end{equation}
By composing $\Cal{T}$ in \eqref{eqn:xv2} with $\Proj$ in \eqref{eqn:proj2}, we obtain
an injective $G$-intertwining operator
\begin{equation}\label{eqn:Tmap}
\wT\colon V(m\varpi_1 + k \varpi_2) 
\longhookrightarrow 
I(m, -\tfrac{1}{2}(m+2k))^{m+k},
\quad v \mapsto (\Proj \circ \Cal{T})(v).
\end{equation}

\begin{prop}\label{prop:TC}
We have 
\begin{equation*}
\Im \wT \subset \Ker \left(\Cal{C}^{m+k+1}_{m,k+1} \right)^{m+k}.
\end{equation*}
\end{prop}

\begin{proof}
The proposition can be shown 
in the same line with the proof for \cite[Thm.\ 6.5]{KuOr25a}.
Indeed, let $d\sigma$ denote the differential of the representation $\sigma$ 
of $G$ on $V(m\varpi_1+k\varpi_2)$.
Then the representation $d\pi_\sigma$ of
$\fg=\fn_-\oplus \fl \oplus \fn_+$ with $\fl:=\fm\oplus \fa$ 
on the non-compact picture of the induced representation
$T_P(m\varpi_1 + k \varpi_2)  \subset C^\infty(N_-)\otimes V(m\varpi_1+k\varpi_2)$
is given by 
\begin{equation}\label{eqn:dpi}
d\pi_\sigma(X) F(\bar{n})= 
d\sigma((\Ad(\bar{n}^{-1})X)_{\fl})F(\bar{n})-(dR((\Ad(\cdot^{-1})X)_{\fn_-})F)(\bar{n})
\end{equation}
for $X \in \fg$ and $F(\bar{n})\in C^\infty(N_-)\otimes  V(m\varpi_1 + k \varpi_2)$.
Here $(\cdot)_{\fl}$ and $(\cdot)_{\fn_-}$ denote the projections 
$\fg \twoheadrightarrow \fl$ and 
$\fg \twoheadrightarrow \fn_-$, respectively,
and $dR$ is the infinitesimal right translation (see, for instance, \cite[(2.4)]{KuOr25a}). 
In particular, as $\fn_-$ is abelian, for $U \in \fn_-$,  we have 
\begin{equation*}
d\pi_\sigma(U)F(\bar{n})=-dR(U)F(\bar{n}).
\end{equation*}
Further, a direct computation shows that
\begin{equation}\label{eqn:dpiN}
d\pi_\sigma(N^-_j)F(x_1, x_2) = -\frac{\partial}{\partial x_j}F(x_1, x_2)
\end{equation}
in the local coordinates \eqref{eqn:coord}
for $F(x_1, x_2) \in C^\infty(\R^2) \otimes V(m\varpi_1+k\varpi_2)$,
where $N_j^-$ $(j=1,2)$ are the vectors of $\fn_-$ defined in \eqref{eqn:Npm}.

Now take a lowest weight vector $v_\ell \in V(m\varpi_1 + k \varpi_2)$.
As $N_-$ acts on $v_\ell$ trivially, we have 
$d\pi_\sigma(N_j^-)\Cal{T}(v_\ell)=0$
in the non-compact picture of 
$T_P(m\varpi_1 + k \varpi_2)$.
It then follows from \eqref{eqn:dpiN} that
$\frac{\partial}{\partial x_j} v_\ell = 0$ for $j=1, 2$
in the local coordinates \eqref{eqn:coord}. Thus, $\Cal{T}(v_\ell)$ is a constant
function in the non-compact picture of $T_P(m\varpi_1 + k \varpi_2)$;
consequently, so is
$\wT(v_\ell) \in C^\infty(N_-) \otimes  W[-\frac{1}{2}(m+2k)]$ 
in the non-compact picture of $I(m, -\tfrac{1}{2}(m+2k))^{m+k}$.
 It then follows from \eqref{eqn:20250811} that
$\Cal{C}^{m+k+1}_{m,k+1}\wT(v_\ell)=0$. Therefore,
$\wT(v_\ell) \in \Ker(\Cal{C}^{m+k+1}_{m,k+1})$.
Now the proposed assertion follows from the fact that
$\wT( V(m\varpi_1 + k \varpi_2)) = \Cal{U}(\fn_+)\wT(v_\ell)$
and the $G$-invariance of $\Ker(\Cal{C}^{m+k+1}_{m,k+1})$.
\end{proof}

\begin{rem}
The image $\wT(V(k \varpi_2))$ for $m=0$ is explicitly given
 in \cite[Prop.\ 8.5]{Kubo24+}.
\end{rem}

Now we are ready to show Theorem \ref{thm:BGGres}.

\begin{proof}[Proof of Theorem \ref{thm:BGGres}]
It follows from Theorem \ref{thm:CPRV} and Proposition \ref{prop:TC} that
the sequence \eqref{eqn:BGGres} is a complex. Now one can compare 
\eqref{eqn:BGGres} with the one in \cite[Sect.\ 4]{EG11} to conclude that 
it is in fact a resolution.
\end{proof}

A similar idea to the proof of Proposition \ref{prop:TC} shows the following.

\begin{prop}\label{prop:KI}
We have
\begin{enumerate}
\item[\emph{(1)}] $\Ker(\Cal{C}^{m+k+1}_{m,k+1})^{m+k+1}=\{0\}$,
\vskip 0.1in

\item[\emph{(2)}] $\Im(\Cal{C}^{m+k+1}_{m,k+1})^{m} 
\simeq I(m, -\tfrac{1}{2}(m+2k))^{m+k+1}$.
\end{enumerate}
\end{prop}

\begin{proof}
The second statement follows from the first as the sequence
\begin{equation*}
0\To
I(m, -\tfrac{1}{2}(m+2k))^{m+k+1}
\stackrel{\Cal{C}^{m+k+1}_{m,k+1}}{\To}
I(m+k+1, \tfrac{1}{2}(3-m+k))^{m}
\end{equation*}
is exact. Thus, it suffices to show the first assertion.

Assume the contrary, namely,
$\Ker \left(\Cal{C}^{m+k+1}_{m,k+1} \right)^{m+k+1}\neq \{0\}$. 
Observe that as the Cartan operator $\Cal{C}^{m+k+1}_{m,k+1}$ is 
obtained from a Cartan component, it is an elliptic operator 
(cf.\ \cite{KOPWZ96}).
Thus, the kernel $\Ker(\Cal{C}^{m+k+1}_{m,k+1})^{m+k+1}
\subset I(m, -\tfrac{1}{2}(m+2k))^{m+k+1}$ 
is finite-dimensional. In particular, the $G$-representation
$\Ker(\Cal{C}^{m+k+1}_{m,k+1})^{m+k+1}$ has a non-zero lowest weight vector $v_0$.

Let $\tau$ denote the $MA$-representation on the fiber 
$W:=(\C_{m+k+1} \otimes \poly^{m}_2) \boxtimes \C_{-\tfrac{1}{2}(m+2k)}$
of $I(m, -\tfrac{1}{2}(m+2k))^{m+k+1}$. Then, as in  \eqref{eqn:dpi},
the representation $d\pi_\tau$ of $\fg$ on the non-compact picture of 
$I(m, -\tfrac{1}{2}(m+2k))^{m+k+1} \subset C^\infty(N_-)\otimes W$
is given by 
\begin{equation*}
d\pi_\tau(X) F(\bar{n})= 
d\tau((\Ad(\bar{n}^{-1})X)_{\fl})F(\bar{n})-(dR((\Ad(\cdot^{-1})X)_{\fn_-})F)(\bar{n})
\end{equation*}
for $X \in \fg$ and $F(\bar{n})\in C^\infty(N_-)\otimes W$.
In particular, for $Z \in \fm\oplus \fa$,
we have
\begin{equation*}
d\pi_\tau(Z)F(\bar{n})=d\tau(Z)F(\bar{n})-(dR(\Ad(\cdot^{-1})Z-Z)F)(\bar{n}).
\end{equation*}
The same arguments as in the proof of Proposition \ref{prop:TC} shows that 
the lowest weight vector $v_0$ is a constant function in the non-compact picture. 
Thus, $d\pi_\tau(Z)v_0(\bar{n})=d\tau(Z)v_0(\bar{n})$.
As $v_0(\bar{n})$ being a lowest weight vector in $I(m, -\tfrac{1}{2}(m+2k))^{m+k+1}$,
we have $d\tau(Z^-)v_0(\bar{n})=d\pi_\tau(Z^-)v_0(\bar{n})=0$ for all
negative root vectors $Z^-\in \fm$. Thus, $v_0(\bar{n})$ can be thought of as a lowest weight vector
of $\poly^{m}_2 \boxtimes \C_{-\tfrac{1}{2}(m+2k)}$. 
Since the weight of $v_0(\bar{n})$ is the same as the lowest weight of $V(m\varpi_1+k\varpi_2)$,
this shows that $d\pi_\tau(\Cal{U}(\fg))v_0(\bar{n}) \simeq V(m\varpi_1+k\varpi_2)$
as $G$-modules;
in particular, 
$d\pi_\tau(\Cal{U}(\fm\oplus \fa))v_0(\bar{n})
\simeq W[-\tfrac{1}{2}(m+2k)]$
as $MA$-modules. 

Now, let $\delta_\tau$ denote the action of $MA$
on $C^\infty(N_-)\otimes W$ induced by the left translation 
on $I(m, -\tfrac{1}{2}(m+2k))^{m+k+1}$, that is,  
for $F(\bar{n}) \in C^\infty(N_-)\otimes W$, we have
\begin{equation*}
\delta_\tau(\ell)F(\bar{n})=\tau(\ell)F(\ell^{-1}\bar{n}\ell).
\end{equation*}
As $v_0(\bar{n})$ is a constant function, 
we have $v_0(\ell^{-1} \bar{n}\ell) = v_0(\bar{n})$ for $\ell \in MA$.
Thus,
\begin{align*}
\delta_\tau(\ell)d\pi_\tau(Z)v_0(\bar{n})
&=\tau(\ell)d\pi_\tau(Z)v_0(\ell^{-1}\bar{n}\ell)\\
&=\tau(\ell)d\tau(Z)v_0(\ell^{-1}\bar{n}\ell)\\
&=\tau(\ell)d\tau(Z)v_0(\bar{n})\\
&=\tau(\ell)d\pi_\tau(Z)v_0(\bar{n})
\end{align*}
for $Z \in \fm\oplus \fa$ and $\ell \in MA$,
which yields that
$d\pi_\tau(\Cal{U}(\fm\oplus \fa))v_0(\bar{n})\ \simeq W$ as $MA$-modules.
Therefore, we have the isomorphisms of $M$-representations
\begin{equation*}
\C_{m+k+1}\otimes \poly_2^m \simeq 
W\vert_M\simeq d\pi_\tau(\Cal{U}(\fm\oplus \fa))v_0(\bar{n})
\simeq W[-\tfrac{1}{2}(m+2k)]\vert_M
\simeq \C_{m+k}\otimes \poly_2^m,
\end{equation*}
which is a contradiction. Now the proposed assertion follows.
\end{proof}

%%%%%%%%%%%%%%%%%%%%%%%%%%%%%%%%%%%%%%%%%%%%%%%
\section{Unitary highest weight modules of $SU(1,2)$}\label{sec:SU}

The aim of the last section is to exploit the results from the previous section to classify 
the irreducible unitary highest weight modules of $SU(1,2)$ at the (first) reduction 
points. This is achieved in Theorem \ref{thm:UHW}.

%%%%%%%%%%%%%%%%%%%%%%%%%%%%%%%%%%%%%%%%%%%%
\subsection{Preliminaries}
Let $G_u:=SU(1,2)$. The maximal compact subgroup $K$ is 
$K=S(U(1)\times U(2)) \simeq U(2)$. 
The complexified Lie algebra of $G_u$ is $\fg = \f{sl}(3,\C)$ and that of 
$K$ is $\fk=\fm\oplus \fa\simeq \f{gl}(2,\C)$. 
The Cartan decomposition of the complex Lie algebra
$\fg$ is then given by $\fg = \fk \oplus \fs$ with $\fs=\fn_+ \oplus \fn_-$. 
Let $G_\C$ and $K_\C$ denote the complexification of $G_u$ and $K$, respectively.
Then $G_\C = SL(3,\C)$ and $K_\C= GL(2,\C)$.
Write 
$N^{\pm}_\C$ for the analytic subgroup of $G_\C$
 with Lie algebra $\fn_\pm$.
Then  $\bar{P}_\C:=K_\C N^-_\C$ is a parabolic subgroup of $G_\C$ with
Lie algebra $\bar{\fp}=\fk \oplus \fn_-=\fm\oplus\fa\oplus \fn_-$. Via the Borel imbedding,
we realize $G_u/K$ as an open subset of $N^+_\C \bar{P}_\C$, which is open dense in $G_\C/\bar{P}_\C$ (cf.\ \cite{Wolf72}): 
\begin{equation*}
G_u/K \stackrel{\text{open}}{\subset} N^+_\C \bar{P}_\C 
\stackrel[\text{dense}]{\text{open}}{\subset} G_\C/\bar{P}_\C.
\end{equation*}

A holomorphic representation $V$ of $K_\C$ is regarded as 
a $P_\C:=K_\C N^+_\C$ and $\bar{P}_\C$-representation by letting 
$N_\C^\pm$ act trivially. Then we write
$\Cal{V}_\C:=G_\C\times_{\bar{P}_\C}V \to G_\C/\bar{P}_\C$ for the $G_\C$-equivariant
holomorphic vector bundle over $G_\C/\bar{P}_\C$. The restriction to the open set $G_u/K$ defines a holomorphic vector bundle $\Cal{V}_{K}:=G_u \times_K V$ over $G_u/K$.
We then obtain a $G_\C$-representation 
$\Cal{O}(G_\C/\bar{P}_\C,\Cal{V}_\C)$ and 
$G_u$-representation $\Cal{O}(G_u/K,\Cal{V}_K)$ 
on the spaces of holomorphic sections
on $G_\C/\bar{P}_\C$ and $G_u/K$, respectively.

In the theory of Davidson--Enright--Stanke (\cite{DES90, DES91}), 
it plays a role that
the $\fk$-representation $V$ of the holomorphic
vector bundle $\Cal{V}_\C=G_\C\times_{\bar{P}_\C}V$ is also
the inducing $\fk$-representation of a generalized Verma module. Indeed, for the space 
$\Cal{O}(G_\C/\bar{P}_\C,\Cal{V}_\C)_K$ of $K$-finite vectors 
of $\Cal{O}(G_\C/\bar{P}_\C,\Cal{V}_\C)$, we have
\begin{align*}
\Cal{O}(G_\C/\bar{P}_\C,\Cal{V}_\C)_K 
= \Pol(\fn_+)\otimes V
\simeq S(\fn_-)\otimes V
\simeq M_\fp(V).
\end{align*}
In order to apply their theory,
we consider $V = \sym_2^m \boxtimes \C_{-\lambda}$ 
instead of $\poly_2^m \boxtimes \C_{\lambda}$
for $\Cal{V}_\C$ so that $V$ is the same as the inducing $\fk$-representation
of the generalized Verma module in the previous sections, where
$\sym_2^m \simeq (\poly_2^m)^\vee \simeq \poly_2^m$.
Then, as in \eqref{eqn:Ind}, for $V = \sym_2^m \boxtimes \C_{-\lambda}$
we put
\begin{align*}
\Cal{O}_{\bar{P}_\C}(m,\lambda):= \Cal{O}(G_\C/\bar{P}_\C, \Cal{V}_\C)
\quad
\text{and}
\quad
\Cal{O}_{K}(m,\lambda):= \Cal{O}(G_u/K, \Cal{V}_K).
\end{align*}

For the rest of this section, we resume the notation defined in 
Sections \ref{sec:Step4} and \ref{sec:main6}.
For instance, the Cartan subalgebra $\fh$ of $\fg$ consists of trace-free 
diagonal matrices and
 $\varpi_1$ denotes the first fundamental weight.
As in \eqref{eqn:Pell}, we write 
$\mathbf{P}^+_{\fk}$ for the set of $\fk$-dominant integral weights.
The sets of positive compact roots $\gD^+_c$ 
and positive noncompact roots $\gD^+_n$
are 
\begin{equation*}
\gD^+_c=\{\eps_2-\eps_3\}
\quad 
\text{and}
\quad
\gD^+_n=\{\eps_1-\eps_2, \eps_1-\eps_3\}.
\end{equation*}

%%%%%%%%%%%%%%%%%%%%%%%%%%%%%%%%%%%%%%%%%%%%
\subsection{Classification of irreducible unitary highest weight modules at reduction points}
For $\mu \in \mathbf{P}^+_{\fk}$, put
$\mathbb{L}(\mu):=
\{\mu+z\varpi_1:z \in \R\}$.
According to the classification theory of Enright--Howe--Wallach \cite{EHW83},
there exists $\mu_0 \in \mathbb{L}(\mu)$ such that  
the irreducible highest weight module $L(\mu_0)$
corresponds to a limit of discrete series representation. We then put
\begin{equation*}
\mathbb{L}(\mu_0)_u:=\{\mu_0+z\varpi_1 \subset  \mathbb{L}(\mu_0):
\text{$L(\mu_0+z\varpi_1)$ is unitarizable}\}.
\end{equation*}
It follows from the \cite{EHW83} that $\mathbb{L}(\mu_0)_u$ takes the following 
form:

\vskip 0.3in

%%%%%%%%%%%%%%%%%%%%%%%%%%%%%%%%%%%%%%

%\begin{figure}[H]
%\caption{$\mathbb{L}(\mu_0)$}
\begin{center}
\begin{tikzpicture}[xscale = 1]    
\tikzset{axes/.style={}}
 
\draw[black, name path = para] plot[domain = -2:4, samples = 100]
({\x}, {3});
  
\draw[black, fill = black] (4, 3) circle[radius = .07cm];

\node[font = \normalsize, below, black] at (4, 3) {\(\text{A}\)};
\end{tikzpicture}
\end{center}
%\end{figure}

%%%%%%%%%%%%%%%%%%%%%%%%%%%%%%%%%%%%%%

The right end point A is called the first reduction point.
It is remarked that since $G_u = SU(1,2)$ has split rank one, there exists only one 
reduction point for each line. Then we put
\begin{align*}
&\widehat{(G_u)}_{h,r}\\
&:=
\{\text{irreducible unitary highest weight modules of $G_u$  
at the (first) reduction points}\}/\sim.
\end{align*}

Now observe that, 
by the same arguments in Section \ref{sec:wT},
one can define the injective $G$-intertwining operator
\begin{equation}\label{eqn:TO}
\wT\colon V(m\varpi_1 + k \varpi_2) \hookrightarrow \Cal{O}_{\bar{P}_\C}(m, \tfrac{1}{2}(m+2k)).
\end{equation}
Here, as we consider $(-\lambda d\chi, \C)$ instead of $(\lambda d\chi, \C)$
for the fiber of the induced representation, we have
$\frac{1}{2}(m+2k)$ in place of $-\frac{1}{2}(m+2k)$  in \eqref{eqn:TO}.
Then, as in \eqref{eqn:BGGres}, we have 
an exact sequence
\begin{equation*}
\begin{aligned}
0\To
V(m\varpi_1 + k \varpi_2)
&\stackrel{\wT}{\To}
\Cal{O}_{\bar{P}_\C}(m, \tfrac{1}{2}(m+2k))\\
&\stackrel{\Cal{C}^{m+k+1}_{m,k+1}}{\To}
\Cal{O}_{\bar{P}_\C}(m+k+1, -\tfrac{1}{2}(3-m+k))\\[5pt]
&\stackrel{\Cal{P}^{k}_{m+k+1,m+1}}{\To}
\Cal{O}_{\bar{P}_\C}(k, -\tfrac{1}{2}(6+2m+k)))
\To 0,
\end{aligned}
\end{equation*}
which gives rise to an exact sequence for $\Cal{O}_K(m,-\lambda)$:
\begin{equation}\label{eqn:BGGgeneral3}
\begin{aligned}
0\To
V(m\varpi_1 + k \varpi_2)
&\stackrel{\wT}{\To}
\Cal{O}_{K}(m, \tfrac{1}{2}(m+2k))\\
&\stackrel{\Cal{C}^{m+k+1}_{m,k+1}}{\To}
\Cal{O}_{K}(m+k+1, -\tfrac{1}{2}(3-m+k))\\[5pt]
&\stackrel{\Cal{P}^{k}_{m+k+1,m+1}}{\To}
\Cal{O}_{K}(k, -\tfrac{1}{2}(6+2m+k)))
\To 0,
\end{aligned}
\end{equation}
where the differential operators
$\Cal{C}^{m+k+1}_{m,k+1}$ and $\Cal{P}^{k}_{m+k+1,m+1}$
are defined as in Section \ref{sec:construction} via the local coordinates
\begin{equation}\label{eqn:coord2}
\C^2 
\stackrel{\sim}{\To} N^+_\C, \quad (z_1, z_2) 
\mapsto \exp(z_1 N^+_1 +z_2N^+_2).\\[3pt]
\end{equation}
We remark that the orders of 
$\Cal{C}^{m+k+1}_{m,k+1}$ and $\Cal{P}^{k}_{m+k+1,m+1}$
are $k+1$ and $m+1$, respectively. 

Let
$\Ker_{\Cal{O}}(\Cal{C}^{m+k+1}_{m,k+1})$ denote
the kernel of 
$\Cal{C}^{m+k+1}_{m,k+1}$ as an operator in 
\eqref{eqn:BGGgeneral3}. The spaces
 $\Im_{\Cal{O}}(\Cal{C}^{m+k+1}_{m,k+1})$ and 
$\Ker_{\Cal{O}}(\Cal{P}^{k}_{m+k+1,m+1})$ are defined, similarly.
Then, combining the resolution \eqref{eqn:BGGgeneral3}
with the theory of Davidson--Enright--Stanke \cite{DES90, DES91}
on covariant differential operators, 
the set $\widehat{(G_u)}_{h,r}$ can be classified as follows.

\begin{thm}\label{thm:UHW}
We have
\begin{align}
\widehat{(G_u)}_{h,r}
&=
\{ \Ker_{\Cal{O}}(\Cal{C}^{1}_{0,1})\}
\cup
\{ \Ker_{\Cal{O}}(\Cal{P}^{k}_{k+1,1}):k \in \Z_{\geq 0}\}\label{eqn:CPU}\\[3pt]
&=
\{ \Ker_{\Cal{O}}(\Cal{C}^{1}_{0,1})\}
\cup
\{ \Im_{\Cal{O}}(\Cal{C}^{k+1}_{0,k+1}):k \in \Z_{\geq 0}\}. \label{eqn:CPU2}
\end{align}
\end{thm}

\begin{proof}
It follows from \cite[Thm.\ 3.2]{DES90} and \cite{EJ90} that
the irreducible unitary highest weight modules $\pi \in 
\widehat{(G_u)}_{h,r}$ are classified by the kernel of PRV operators
whose orders are equal to the level of reduction of $\pi$.
In the present case, the level of reduction is always 1 
as there is only one reduction point.
Since the order of the PRV operator $\Cal{P}^{k}_{m+k+1,m+1}$ is $m+1$, 
we have $\Ker_{\Cal{O}}(\Cal{P}^{k}_{k+1,1}) \in \widehat{(G_u)}_{h,r}$
for all $k \in \Z_{\geq 0}$.
As remarked in the beginning of Section \ref{sec:comp},
the operators $\Cal{C}^{k+1}_{0,k+1}$ are also PRV operators,
which implies that $\Ker_{\Cal{O}}(\Cal{C}^{1}_{0,1}) \in  \widehat{(G_u)}_{h,r}$.
Since the right hand side of \eqref{eqn:CPU} exhaust all PRV operators with order 1,
the identity \eqref{eqn:CPU} follows.

The second line \eqref{eqn:CPU2} is the direct consequence of
\eqref{eqn:CPU} and the following resolution:
\begin{equation}\label{eqn:BGGgeneral4}
\begin{aligned}
0\To
V(k \varpi_2)
&\stackrel{\wT}{\To}
\Cal{O}_K(0, k)\\
&\stackrel{\Cal{C}^{k+1}_{0,k+1}}{\To}
\Cal{O}_K(k+1, -\tfrac{1}{2}(3+k))\\[5pt]
&\stackrel{\Cal{P}^{k}_{k+1,1}}{\To}
\Cal{O}_K(k, -\tfrac{1}{2}(6+k))
\To 0.
\end{aligned}
\end{equation}
This concludes the theorem.
\end{proof}

%There are two remarks in order.

%%%%%%%%%%%%%%%%%%%%%%%%%%%%%%%%%%%%%%%%%%%%%%%
\begin{rem}\label{rem:78}
For $\pi \in \widehat{(G_u)}_{h,r}$, denote by $\pi_K$ the space of $K$-finite vectors of $\pi$.
Then the highest weights of 
$\Ker_{\Cal{O}}(\Cal{C}^{1}_{0,1})_K$ and 
$\Ker_{\Cal{O}}(\Cal{P}^{k}_{k+1,1})_K$ are given as follows.

If $m=0$, then the BGG resolution \eqref{eqn:BGGgeneral2} of generalized Verma modules
reads
\begin{equation*}
\begin{aligned}
0
\To M_\fp(k, -\tfrac{1}{2}(6+k))
&\stackrel{\varphi^{k}_{k+1,1}}{\To}
M_\fp(k+1, -\tfrac{1}{2}(3+k))\\[5pt]
&\stackrel{\varphi^{k+1}_{0,k+1}}{\To}
M_\fp(0, k)\\[5pt]
&\stackrel{\pr}{\To}
V(k\varpi_1)
\To 0.
\end{aligned}
\end{equation*}
By \cite[Thms.\ 2.4 and 3.2]{DES90}, we have 
\begin{align*}
\Ker_{\Cal{O}}(\Cal{C}^{1}_{0,1})_K
&\simeq 
M_\fp(0,0)/\Im (\varphi^1_{0,1}) \simeq \C_{\triv}, \\[3pt]
\Ker_{\Cal{O}}(\Cal{P}^{k}_{k+1,1})_K
&\simeq
M_\fp(k+1,-\tfrac{1}{2}(3+k))/\Im (\varphi^k_{k+1,1}).
\end{align*}
Then, in the coordinates used in Section \ref{sec:Step4}, the following hold.

\begin{enumerate}[(1)]
\item The highest weight of 
$\Ker_{\Cal{O}}(\Cal{C}^{1}_{0,1})_K$ is $(0,0,0)$.
\vskip 0.1in

\item The highest weight of $\Ker_{\Cal{O}}(\Cal{P}^{k}_{k+1,1})_K$ is
\begin{align*}
(k+1)\omega_1-\frac{1}{2}(3+k)d\chi = (-1,1,0)+\frac{k}{3}(-1,2,-1),
\end{align*}
where $\omega_1=\frac{1}{2}(0,1,-1)$ and $d\chi=\frac{1}{3}(2,-1,-1)$.
\end{enumerate}
\end{rem}

%%%%%%%%%%%%%%%%%%%%%%%%%%%%%%%%%%%%%%%%%%%%%%%

\begin{rem}
It follows from \cite{EHW83} that 
a root system $Q$ is associated to the first reduction point.
(Also, see \cite[Table 6.22]{DES90}).
In the present case, there are only two possibilities on $Q$, 
namely, $Q=SU(1,1)$ or $SU(1,2)$.
The sets $\{ \Ker_{\Cal{O}}(\Cal{C}^{1}_{0,1})\}$ 
and $\{ \Ker_{\Cal{O}}(\Cal{P}^{k}_{k+1,1}):k \in \Z_{\geq 0}\}$
correspond to the cases $Q=SU(1,2), SU(1,1)$, respectively. 
Here are some details.

In \cite{DES90}, the highest weights of $\pi_K$
for $\pi \in \widehat{(G_u)}_{h,r}$ are given by 
\begin{equation*}
\lambda_{q'} + C_{q'},
\end{equation*}
which is associated to the root system $Q=SU(1,q')$ for $q'=1,2$
(\cite[Def.\ 6.5 and Prop.\ 6.6]{DES90}).
Here, $\lambda_{q'}$ is 
\begin{equation*}
\lambda_{q'}:=-(3-q')\varpi_1+\varpi_{3-q'}
\end{equation*}
and $C_{q'}$ is an integral cone of $\fk$-dominant integral weights defined by
\begin{equation*}
C_{q'}:=\{k (-\varpi_1 + \varpi_{3-q'}): k \in \Z_{\geq 0}\}
\end{equation*}
(see \cite[Tables 6.21 and 6.22]{DES90} for the full description). 
Now we consider the cases $Q=SU(1,2), SU(1,1)$, separately.

\begin{enumerate}[(1)]

\item $Q=SU(1,2)$: In this case, we have $\lambda_2=(0,0,0)$ and $C_2=\{(0,0,0)\}$. Thus,
the corresponding highest weight is $(0,0,0)$. By Remark \ref{rem:78} (1),
this corresponds to $\Ker_{\Cal{O}}(\Cal{C}^{1}_{0,1})$.

\vskip 0.1in

\item $Q=SU(1,1)$: In this case, we have 
\begin{equation*}
\lambda_1=-2\varpi_1+\varpi_2=(-1,1,0)
\end{equation*}
and 
\begin{equation*}
C_1=\{k(-\varpi_1+\varpi_2):k\in \Z_{\geq 0}\}=\{\tfrac{k}{3}(-1,2,-1):k\in \Z_{\geq 0}\}.
\end{equation*}
So, the highest weights have the form
\begin{equation*}
(-1,1,0)+\tfrac{k}{3}(-1,2,-1),
\end{equation*}
which is the highest weight of $\Ker_{\Cal{O}}(\Cal{P}^{k}_{k+1,1})_K$ by Remark \ref{rem:78} (2).

\end{enumerate}
\end{rem}

%%%%%%%%%%%%%%%%%%%%%%%%%%%%%%%%%%%%%%%%%%%%%%%
\noindent
\textbf{Acknowledgements.}
The first author wishes to thank Boris Doubrov, Dennis The, and Tohru Morimoto 
for fruitful discussions concerning this work.
The authors would also like to express their gratitude to 
Anthony Kable and Toshiyuki Kobayashi for their valuable comments.
Finally, they extend their appreciation to the anonymous referee 
 for a careful review of the manuscript.

The first author was partially supported by JSPS
Grant-in-Aid for Scientific Research(C) (JP22K03362).

%%%%%%%%%%%%%%%%%%%%%%%%%%%%%%%%%%%%%%%%%%
%% References

\bibliographystyle{amsplain}

%%%%%%%%%%%%%%%%%%%%%%%%%%%%%%%%%%%%%%%%%%

\end{document}